\documentclass[a4paper,11pt, reqno]{amsart}
\usepackage{amssymb,amsthm,amsmath}
 \usepackage{enumerate}
\usepackage{ifthen}
\usepackage{graphicx}
\usepackage{cite}
\usepackage{amsmath,amscd,amssymb}
\usepackage{latexsym}
\usepackage[colorlinks,citecolor=blue,pagebackref,hypertexnames=false]{hyperref}

\numberwithin{equation}{section}
\allowdisplaybreaks[2]
\theoremstyle{plain}
\newtheorem{theorem}{Theorem}[section]

\newtheorem{lemma}[theorem]{Lemma}
\newtheorem{corollary}[theorem]{Corollary}
\newtheorem{proposition}[theorem]{Proposition}

\theoremstyle{definition}
\newtheorem{definition}[theorem]{Definition}

\theoremstyle{remark}
\newtheorem{remark}[theorem]{Remark}

\newtheorem{case[theorem]}{Case}

\def\supp{\hbox{supp\,}}
\def\norm#1.#2.{\lVert#1\rVert_{#2}}

\title[Decay estimates for a class of Dunkl wave equations]{Decay estimates for a class of Dunkl wave equations
}

\author{Cheng Luo}
\author{Shyam Swarup Mondal} 
\author{Manli Song}

\address{\endgraf School of Mathematics and Statistics, Northwestern Polytechnical University, Xi'an, Shaanxi 710129, China}
\email{caashmale@163.com}

\address{ 
\endgraf Stat-Math unit,
Indian Statistical Institute Kolkata, 
BT Road,  Baranagar, Kolkata  700108, India}
\email{mondalshyam055@gmail.com}

\address{\endgraf School of Mathematics and Statistics, Northwestern Polytechnical University, Xi'an, Shaanxi 710129, China}
\email{mlsong@nwpu.edu.cn}

\date{\today}
\begin{document}
	
	\maketitle

	\allowdisplaybreaks

\begin{abstract}
Let $\Delta_\kappa$ be  the Dunkl Laplacian on $\mathbb{R}^n$ and $\phi: \mathbb{R}^+ \to \mathbb{R}$ is a smooth function. The aim of this manuscript is twofold. 	First, we study the decay estimate for a class of dispersive semigroup of the form $e^{it\phi(\sqrt{-\Delta_\kappa})}$. 
We overcome the difficulty arising from the non-homogeneousity of $\phi$ by frequency localization. 
As applications,  in the next part of the paper,   we establish Strichartz estimates for some concrete wave equations associated with the Dunkl Laplacian $\Delta_k,$ which corresponds to  $\phi(r)=r, r^2, r^2+r^4, \sqrt{1+r^2}, \sqrt{1+r^4}$, and $r^\mu,0<\mu\leq 2, \mu\neq 1$. More precisely, we unify and simplify all the known dispersive estimates and
extend to more general cases.
 Finally,  using the decay estimates, we prove the global-in-time existence of small data Sobolev solutions for the nonlinear Klein-Gordon equation and beam equation with the power type nonlinearities.

\end{abstract}
	\tableofcontents

	\section{Introduction}
The main aim of this paper is to investigate  decay estimates for a class of dispersive equations of the form	\begin{equation}\label{HSEquation}
	\begin{cases}
	i\partial_tu+\phi(\sqrt{-\Delta_\kappa}) u=f,\\
	u|_{t=0}=u_0,
	\end{cases}
	\end{equation}
where $\Delta_\kappa$ is the Dunkl Laplacian on $\mathbb{R}^n$, $f:\mathbb{R}^n\times \mathbb R\to \mathbb C$, and $\phi: \mathbb{R}^+ \to \mathbb{R}$ is a smooth function that satisfies some growth condition.

Dispersive inequalities for evolution equations, such as the Schr\"odinger and wave equations play a significant role in the study of semilinear and quasilinear problems, which naturally appear in several physical applications. Proving dispersion means estimating the decay in time on the evolution group associated with the free equation.  In recent times, dispersive properties of evolution equations have become a crucial tool in the study of a variety of questions, including local and global existence for nonlinear equations, well-posedness of Cauchy problems
for nonlinear equations in Sobolev spaces of low order, scattering theory, and many others. In many circumstances, the essential step in proving this time decay is based on the application of a stationary phase theorem on an (approximate) representation of the solution.  Dispersion phenomena, when combined with an abstract functional analysis argument, famously known as the $TT^*$-argument (see \cite{KT}), offer a range of estimates involving space-time Lebesgue norms. These inequalities also known as Strichartz estimates, have evolved into a fundamental and amazing tool in the study of nonlinear partial differential equations over the last few decades, see \cite{Bahouri-app} and references therein.

 Strichartz estimates also be interpreted as Fourier restriction estimates, which have a vital role in classical harmonic analysis and are closely linked to arithmetic combinatorics. Considerable attention has been devoted by several researchers to deriving Strichartz estimates in different frameworks. For example, in the Euclidean setting,  these estimates have been proved for several dispersive equations, such as the wave equation and Schr\"{o}dinger equation, for instance, see the pioneering works \cite{GV, KT, Str}. The theory of Strichartz estimates has also been extended in the non-Euclidean frameworks. 
For example, see \cite{BGX2000, H2005, FMV1, FV,  LS2014, SZ, Song2016,  BKG, BBG2021, SY2023} for Strichartz estimates on H-type groups and more generality of step $2$ stratified Lie groups, \cite{BDDM2019, DPR2010, FSW, NR2005, R2008, S2013, Mejjaoli2008-1, Mejjaoli2009, Mejjaoli2013, Ben said, Ratna3} for  metric measure spaces, 
\cite{AP2009, AP2014, APV2012} for hyperbolic spaces, \cite{B1993, BGT2004} for compact Riemannian manifolds,  and \cite{ILP2014} for  bounded domains.

One of the essential ingredients in the derivation of Strichartz estimates is the dispersive estimates. Our aim of this paper is to establish a decay estimate for a class of dispersive equations as follows:
	\begin{equation}\label{HSEquation}
	\begin{cases}
	i\partial_tu+\phi(\sqrt{-\Delta_\kappa}) u=f,\\
	u|_{t=0}=u_0,
	\end{cases}
	\end{equation}
where $\Delta_\kappa$ is the Dunkl Laplacian on $\mathbb{R}^n$, $f:\mathbb{R}^n\times \mathbb R\to \mathbb C$, and $\phi: \mathbb{R}^+ \to \mathbb{R}$ is a smooth function satisfying
\\

(C1)~There exists $m_1>0$ such that for any $\alpha \geqslant 2$ and $\alpha \in \mathbb{N}$,
\begin{equation*}
|\phi'(r)| \sim r^{m_1-1}  , \quad |\phi^{(\alpha)}(r)| \lesssim r^{m_1-\alpha},\quad r \geqslant 1.
\end{equation*}

(C2)~There exists $m_2>0$ such that for any $\alpha \geqslant 2$ and $\alpha \in \mathbb{N}$,
\begin{equation*}
|\phi'(r)| \sim r^{m_2-1}  , \quad |\phi^{(\alpha)}(r)| \lesssim r^{m_2-\alpha},\quad 0<r<1.
\end{equation*}

(C3)~There exists $\alpha_1>0$ such that
\begin{equation*}
|\phi''(r)| \sim r^{\alpha_1-2}, \quad r \geqslant1.
\end{equation*}

(C4)~There exists $\alpha_2>0$ such that
\begin{equation*}
|\phi''(r)| \sim r^{\alpha_2-2}, \quad  0<r<1.
\end{equation*}
Here we use a Fourier multiplier notation to  define $\phi(\sqrt{-\Delta_\kappa}) $  as in the following form $$\phi(\sqrt{-\Delta_\kappa})f=\mathcal{F}_\kappa^{-1}\phi(|\xi|) \mathcal{F}_\kappa f, $$
where $\mathcal{F}_\kappa $ denotes Dunkl transform, defined in (\ref{Dunkltransform}), Section \ref{sec2}.  Notice that the conditions (C1) and (C3) represent the homogeneous order of $\phi$ in high frequency, and the conditions (C2) and (C4) reflect the homogeneous order of $\phi$ in low frequency. If $\phi$ satisfies conditions (C1) and (C3), then $\alpha_1\leq m_1$. Similarly, if $\phi$ satisfies (C2) and (C4), then $\alpha_2\geq m_2$.
 
Here, we want to emphasize that several dispersive wave equations associated with Dunkl Laplacian reduced to type (\ref{HSEquation}). 
For instance, the Dunkl Schr\"{o}dinger equation corresponds to $\phi(r)=r^2$, the Dunkl wave equation corresponds to $\phi(r)=r$, the fractional Dunkl  Schr\"{o}dinger equation corresponds to $\phi(r)=r^\mu$, $0<\mu<2,\mu\neq 1$, the fourth-order Dunkl Schr\"{o}dinger equation corresponds to $\phi(r)=r^2+r^4$, the beam equation $\phi(r)=\sqrt{1+r^4}$, and the Klein-Gordon equation corresponds to $\phi(r)=\sqrt{1+r^2}$, among others.  


 Dunkl operators are differential-difference operators that generalize the standard partial derivatives. These operators were introduced by Charles Dunkl \cite{dun1991} using a finite reflection group and a parameter function on its root system. Applications of Dunkl operators in mathematics and mathematical physics are numerous. For example, they play an important role in the integration of quantum many-body problems of the Calogero-Moser-Sutherland type and are employed in the study of probabilistic processes. Additionally, these operators have had a lasting impression on the study of special functions in one and multiple variables as well as orthogonal polynomials,  see \cite{Dunkl book}.

For $\kappa=0,$ as the the Dunkl Laplacian $\Delta_\kappa$ reduces to usual Laplacian $\Delta$ on $\mathbb{R}^n$,    the system (\ref{HSEquation}) reduces to 
\begin{equation}\label{Eucledian}
	\begin{cases}
	i\partial_tu+\phi(\sqrt{-\Delta}) u=f,\\
	u|_{t=0}=u_0.
	\end{cases}
	\end{equation}
  In 1977, Strichartz \cite{Str}  derived the priori estimates for the solution to (\ref{Eucledian}) in the space-time norm using the classical Fourier restriction theorem of Stein and Tomas \cite{Stein1984, T1, T2}. However, later, it was improved by Ginibre-Velo \cite{GV} using a standard duality argument along with the dispersive estimate   of the form 
 \begin{align}\label{dispersive estimate}\left\|e^{i t \phi(\sqrt{-\Delta})} u_0\right\|_X \lesssim|t|^{-\theta}\left\|u_0\right\|_{X^*},\end{align}
where $X^*$ is the dual space of $X$. In 2008, Guo-Peng-Wang \cite{GPW2008} used a unified way to study the decay for a class of dispersive semigroup $e^{it\phi(\sqrt{-\Delta})}$ on $\mathbb{R}^n$  for the system (\ref{Eucledian}).  They assumed that $\phi: \mathbb{R}^+ \to \mathbb{R}$ is smooth satisfying (C1)-(C4) and obtained several decay estimates in time for the dispersive semigroup $e^{it\phi(\sqrt{-\Delta})}$ by introducing Littlewood–Paley projector operator, see  \cite[Theorem 1]{GPW2008}.

Motivated by the recent works, in this paper, we aim to establish the decay estimate for a class of dispersive equations associated with Dunkl Laplacian \eqref{HSEquation}.   Note that, if  $\phi$ is a homogeneous function of order $m$, namely, $\phi(\lambda r)=$ $\lambda^m \phi(r)$ for all $ \lambda>0$, the dispersive estimate can be easily obtained by a theorem of Littman and dyadic decomposition. However, things become complicated when the function $\phi$ is not homogeneous because the scaling constants can not be effectively separated from the time. Since the phases we consider are not necessarily homogeneous or weakly dispersive, to overcome the difficulty, we use a Littlewood-Paley decomposition (see \cite{GPW2008}) to more accurately capture  the
difference between the low and the high frequencies in different scales and catch the possible need for extra regularity.


Let  $S_0$  and $\tilde{\Delta}_j$ be the Littlewood-Paley decomposition related to the Dunkl Laplacian $\Delta_\kappa$ and $N=2\gamma+n$ (see Section \ref{sec2} for detailed definition). Then the main results of this paper are as follows:
\begin{theorem}\label{ResultTime}
	Assume $\phi:\mathbb{R}^+ \to \mathbb{R}$ is smooth and $U_t=e^{it\phi(\sqrt{-\Delta_\kappa})}$.   Then the subsequent outcomes are valid.
 \begin{enumerate}[(A)]
 \item   For $j \geq 0$, assume that $\phi$ satisfies (C1), then
	\begin{equation}\label{res3-2}
	\|U_t \tilde{\Delta}_j u_0\|_{L_\kappa^\infty(\mathbb{R}^n)} \lesssim |t|^{-\theta}2^{j\left(N-m_1\theta\right)}\|u_0\|_{L_\kappa^1(\mathbb{R}^n)},\,0 \leq \theta \leq \frac{N-1}{2}.
	\end{equation}
 In addition, if $\phi$ satisfies (C3), then
	\begin{equation}\label{res3-3}
	\|U_t \tilde{\Delta}_j u_0\|_{L_\kappa^\infty(\mathbb{R}^n)} \lesssim |t|^{-\frac{N-1+\theta}{2}}2^{j\left(N-\frac{m_1(N-1+\theta)}{2}-\frac{\theta(\alpha_1-m_1)}{2}\right)}\|u_0\|_{L_\kappa^1(\mathbb{R}^n)},\,   0 \leq \theta \leq 1.
	\end{equation}
 \item For $j < 0$, suppose $\phi$ satisfies (C2), then
	\begin{equation*}
	\|U_t \tilde{\Delta}_j u_0\|_{L_\kappa^\infty(\mathbb{R}^n)} \lesssim |t|^{-\theta}2^{j\left(N-m_2\theta\right)}\|u_0\|_{L_\kappa^1(\mathbb{R}^n)},\,0 \leq \theta \leq \frac{N-1}{2}.
	\end{equation*}

In addition, if $\phi$ satisfies (C4), then
	\begin{equation*}
	\|U_t \tilde{\Delta}_j u_0\|_{L_\kappa^\infty(\mathbb{R}^n)} \lesssim |t|^{-\frac{N-1+\theta}{2}}2^{j\left(N-\frac{m_2(N-1+\theta)}{2}-\frac{\theta(\alpha_2-m_2)}{2}\right)}\|u_0\|_{L_\kappa^1(\mathbb{R}^n)},\,   0 \leq \theta \leq 1.
	\end{equation*}
\item If $\phi$ satisfies (C2), then
	\begin{equation}\label{res-sum1}
 \|U_t S_0u_0\|_{L_\kappa^\infty(\mathbb{R}^n)} \lesssim (1+|t|)^{-\theta}\|u_0\|_{L_\kappa^1(\mathbb{R}^n)},\,\theta=\min\left(\frac{N}{m_2},\frac{N-1}{2}\right).
	\end{equation}

In addition, if $\phi$ satisfies (C4) and $\alpha_2=m_2$, then
	\begin{equation}\label{res-sum2}
	\|U_t S_0u_0\|_{L_\kappa^\infty(\mathbb{R}^n)} \lesssim (1+|t|)^{-\theta}\|u_0\|_{L_\kappa^1(\mathbb{R}^n)},\,\theta=\min\left(\frac{N}{m_2},\frac{N}{2}\right).
	\end{equation}
	 \end{enumerate}
\end{theorem}
 We prove the above result in Section \ref{sec4}.    The main difficulty in the setup of Dunkl is that we don't have an explicit expression for the Dunkl transform compared to the classical Fourier transform. In this paper, we used Littlewood-Paley decomposition to better understand the regularity is needed at low and high frequencies in order to get a time decay.  We also use the fact that Dunkl transform for radial functions can be expressed in terms of integrals of Bessel functions. In order to estimate these integrals, we rely on Van der Corput's Lemma and the decay properties of the Bessel functions. 
 
 Using the   decay estimates in Theorem \ref{ResultTime},  
 we establish Strichartz estimates for  the Dunkl Schr\"{o}dinger equation (corresponds to $\phi(r)=r^2$), the Dunkl wave equation (corresponds to $\phi(r)=r$), the fractional Dunkl  Schr\"{o}dinger equation (corresponds to $\phi(r)=r^\mu$), $0<\mu<2, \mu\neq 1$, the fourth-order Dunkl Schr\"{o}dinger equation (corresponds to $\phi(r)=r^2+r^4$), the beam equation  (corresponds to $\phi(r)=\sqrt{1+r^4}$), and the Klein-Gordon equation (corresponds to $\phi(r)=\sqrt{1+r^2}$). We also prove the global well-posedness for the nonlinear Klein-Gordon equation and beam equation for small initial data.

 There are many situations where dispersion phenomena fail, for instance in compact manifolds, on some domains, on the Heisenberg group. However, in some cases, Strichartz estimates (in weak forms) or smoothing properties can be established using other approaches that are not based on dispersive inequality techniques. The Fourier restriction theorem and  Kato smoothing effect techniques stand out among them.
Recently, the authors in \cite{shyam}  proved  Strichartz estimates for the Dunkl Schr\"{o}dinger equation (which corresponds to $\phi(r)=r^2$)  using the   Fourier-Dunkl transform restriction theorem for paraboloid surfaces. See also \cite{Mejjaoli2008-1, Mejjaoli2009}.
 \begin{remark}
	Since the Dunkl Laplacian $\Delta_\kappa$ reduces to the Euclidean Laplacian $\Delta$ for $\kappa=0$, we recover the results proved in Guo-Peng-Wang \cite{GPW2008}.
\end{remark}

Apart from the introduction, this paper is organized as follows: In Section \ref{sec2}, we recall harmonic analysis associated with the Dunkl operators, Dunkl transforms, and their certain properties.  We also recall some useful technical lemmas for proving the decay estimates and Strichartz estimates. In Section \ref{sec4}, we prove our main result of this paper, i.e., Theorem \ref{ResultTime}. We provide   Strichartz
estimates for some concrete wave equations associated with the Dunkl Laplacian $\Delta_\kappa$ in Section \ref{sec5}. In Section \ref{sec6}, we prove global well-posedness for the nonlinear Klein-Gordon equation and beam equation for small initial data.\\

\noindent \textbf{Notations:}
\begin{itemize}
\item  $A\lesssim B$ means that $A\leqslant CB$, and $A\sim B$ stands for $C_1B\leq A\leq C_2B$, where $C$, $C_1$, $C_2$ denote positive universal constants.
\item $\Delta_\kappa$ denotes Dunkl-Laplacian on $\mathbb{R}^n$.
\item $ S_0$  and $\tilde{\Delta}_j$ be the Littlewood-Paley decomposition related to  $\Delta_\kappa$.
\item $J_\nu$ denotes the Bessel
function of order $\nu.$ 
\item   $B^{s,\kappa}_{p,q}$ denotes the   Dunkl Besov space.
\end{itemize}

\section{Preliminaries}\label{sec2}
In this section, we recall some basic definitions and important properties of the Dunkl operators and briefly overview the related harmonic analysis to make the paper self-contained. A complete account of harmonic analysis related to Dunkl operators can be found in \cite{ros,ben,Ratna3, Mejjaoli2009,MS,Thangavelu}.  
	\subsection{Dunkl operators}\label{dunklop} The basic ingredient in the theory of Dunkl operators are root systems and finite reflection groups. We start this section by the definition of root system. 
 
 Let $\langle\cdot, \cdot\rangle$  denotes the standard Euclidean scalar product in $\mathbb{R}^{n}$. For $x \in \mathbb{R}^{n}$, we denote   $|x|$ as $|x|=\langle x, x\rangle^{1 / 2}$.
	For $\alpha \in \mathbb{R}^{n} \backslash\{0\},$ we denote $r_{\alpha}$ as  the reflection with respect to the hyperplane $\langle \alpha\rangle^{\perp}$, orthogonal to $\alpha$ and is defined by
	$$
	r_{\alpha}(x):=x-2 \frac{\langle \alpha, x\rangle}{|\alpha|^{2}} \alpha, \quad x \in \mathbb{R}^{n}.
	$$
	A finite set $\mathcal{R}$ in $\mathbb{R}^{n} \backslash\{0\}$ is  said to be a  root system if the following holds:
	\begin{enumerate}
		\item $ r_{\alpha}(\mathcal{R})=\mathcal{R}$ for all $\alpha \in \mathcal{R}$,
		\item $\mathcal{R} \cap \mathbb{R} \alpha=\{\pm \alpha\}$ for all $\alpha \in \mathcal{R}$.
	\end{enumerate}
	For a given root system $\mathcal{R}$, the subgroup $G \subset O(n, \mathbb{R})$ generated by the reflections $\left\{r_{\alpha} \mid \alpha \in \mathcal{R}\right\}$ is called the finite Coxeter group associated with $\mathcal{R}$.  The dimension of $span \mathcal{R}$ is called the rank of $\mathcal{R}$.   For a   detailed study on the theory of finite reflection groups, we refer to  \cite{hum}.  	Let    $\mathcal{R}^+:=\{\alpha\in\mathcal{R}:\langle\alpha,\beta\rangle>0\}$ for some $\beta\in\mathbb{R}^n\backslash\bigcup_{\alpha\in\mathcal{R}}\langle \alpha\rangle^{\perp}$,  be a  fix  positive root system.

	Some typical examples of such a system are the Weyl groups such as the symmetric group $S_{n}$ for the type $A_{n-1}$ root system and the hyperoctahedral group for the type $B_{n}$ root system. In addition, $H_{3}, H_{4}$ (icosahedral groups) and, $I_{2}(n)$ (symmetry group of the regular $n$-gon) are also the Coxeter groups. 

	A multiplicity function for $G$ is a function $\kappa: \mathcal{R} \rightarrow \mathbb{C}$ which is constant on $G$-orbits. Setting $\kappa_{\alpha}:=\kappa(\alpha)$ for $\alpha \in \mathcal{R},$  from the definition of $G$-invariant,  we have $\kappa_{g \alpha}=\kappa_{\alpha}$ for all $g \in G$.  We say $\kappa$ is non-negative if $\kappa_{\alpha} \geq 0$ for all $\alpha \in \mathcal{R}$. The $\mathbb{C}$-vector space of non-negative multiplicity functions on $\mathcal{R}$ is denoted by $\mathcal{K}^{+}$. Let us denote  $\gamma$ as $\gamma=\gamma(\kappa):=\sum\limits_{\alpha\in\mathcal{R}^+}\kappa(\alpha)$. For convenience, we denote $N=2\gamma+n$ throughout the paper. 
 
 For $\xi \in \mathbb{C}^{n}$ and $\kappa \in \mathcal{K}^{+},$ Dunkl  in 1989 introduced a family of first-order differential-difference operators $T_{\xi}:= T_{\xi}(\kappa)$  by
	\begin{align}\label{dunkl}
		T_{\xi}(\kappa) f(x):=\partial_{\xi} f(x)+\sum_{\alpha \in \mathcal{R}^{+}} \kappa_{\alpha }\langle \alpha, \xi\rangle \frac{f(x)-f\left(r_{\alpha} x\right)}{\langle \alpha, x\rangle}, \quad f \in C^{1}\left(\mathbb{R}^{n}\right),
	\end{align}
	where $\partial_{\xi}$ denotes the directional derivative corresponding to $\xi$.  The operator $T_\xi$   defined in (\ref{dunkl}), formally known as Dunkl operator and is one of the most important developments in the theory of special functions associated with root systems \cite{dun}. They commute pairwise and are skew-symmetric with respect to the $G$-invariant measure $h_\kappa(x)dx$, where the weight function $$h_\kappa(x):=\prod\limits_{\alpha\in \mathcal{R}^{+}}|\langle \alpha, x\rangle|^{2 \kappa_{\alpha}}$$   is of 
 homogeneous degree $2\gamma$. Thanks to the $G$-invariance of the multiplicity function, the definition of the Dunkl operator is independent of the choice of the positive subsystem $\mathcal{R}^{+}$.  In \cite{dun1991},	it is shown that for any $\kappa\in\mathcal{K}^+$, there is a unique linear isomorphism $V_\kappa$ (Dunkl's intertwining operator) on the space $\mathcal{P}(\mathbb{R}^n)$ of polynomials on $\mathbb{R}^n$ such that
	\begin{enumerate}
		\item $V_\kappa\left(\mathcal{P}_m(\mathbb{R}^n)\right)=\mathcal{P}_m(\mathbb{R}^n)$ for all $m\in\mathbb{N}$,
		\item $V_\kappa|_{\mathcal{P}_0(\mathbb{R}^n)}=id$,
		\item $T_\xi(\kappa)V_\kappa=V_\kappa\partial_\xi$,
	\end{enumerate}
	where  $\mathcal{P}_m(\mathbb{R}^n)$ denotes the space of homogeneous polynomials of degree $m$. 
	
	For any finite reflection group $G$ and   any $k\in\mathcal{K}^+$, R\"{o}sler in \cite{ros} proved that there exists a unique positive Radon probability measure $\rho_x^\kappa$ on $\mathbb{R}^n$ such that
	\begin{equation}\label{239}	V_\kappa f(x)=\int_{\mathbb{R}^n}f(\xi)d\rho_x^\kappa(\xi).
	\end{equation}
	The measure $\rho_x^k$ depends on $x\in\mathbb{R}^n$ and its support is contained in the ball $B\left(\|x\|\right):=\{\xi\in\mathbb{R}^n: \|\xi\|\leq\|x\|\}$. In view of the Laplace type representation \eqref{239}, Dunkl's intertwining operator $V_\kappa$ can be extended to a larger class of function spaces.
	
	Let $\{\xi_1, \xi_2, \cdots, \xi_n\}$ be an orthonormal basis of $(\mathbb{R}^n , \langle \cdot, \cdot\rangle )$. Then  the Dunkl Laplacian operator $\Delta_\kappa $ is defined as $$\Delta_\kappa=\sum_{j=1}^nT^2_{\xi_j}(\kappa).$$
	The definition of  $\Delta_\kappa$ is independent of the choice of the orthonormal basis of $\mathbb{R}^n.$ In fact,   one can see that the operator $\Delta_\kappa $ also can be expressed as $$\Delta_\kappa f(x)=\Delta f(x)+\sum_{\alpha \in \mathcal{R}^+}\kappa_\alpha \left\{\frac{2 \langle \nabla f(x),  \alpha \rangle}{\langle \alpha, x\rangle } -|\alpha|^2 \frac{f(x)-f(r_\alpha x)}{\langle \alpha, x\rangle^2 } \right\}, \quad f\in C^1(\mathbb{R}^n),$$
	where $\nabla$ and $\Delta$ are    the usual gradiant  and   Laplacian operator on $\mathbb{R}^n$, respectively.  Observe that, for $\kappa \equiv 0$, the Dunkl Laplacian operator $\Delta_\kappa$ reduces to the classical Euclidean Laplacian $\Delta$. 

For every $y\in\mathbb{R}^n$, the system 
\begin{equation*}
	\begin{cases}
	T_\xi u(x,y)=\langle y,\xi\rangle\, u(x,y), \quad \xi\in\mathbb{R}^n,\\
	u(0,y)=1,
	\end{cases}
	\end{equation*}
admits a unique analytic solution on $\mathbb{R}^n$, which we denote as  $K(x,y)$ and is known as the Dunkl kernel in general. The kernel has a unique holomorphic extension to $\mathbb{C}^n\times \mathbb{C}^n$ and is given by
\begin{equation*}
K(x,y)=V_\kappa(e^{\langle \cdot,y\rangle})(x)=\int_{\mathbb{R}^n}e^{\langle \xi,y\rangle}d\rho_x^\kappa(\xi),\quad \forall x,y\in\mathbb{R}^n,
	\end{equation*}
where the positive Radon probability measure     $\rho_x^k$ depends on $x\in\mathbb{R}^n$ and its support is contained in the ball $B\left(\|x\|\right):=\{\xi\in\mathbb{R}^n: \|\xi\|\leq\|x\|\}$. When $\kappa \equiv 0$, the Dunkl kernel $K(x,y)$  reduces to the exponential $e^{\langle x,y\rangle}$. For $x,y\in\mathbb{R}^n$ and  $z,w\in\mathbb{C}^n$, the Dunkl kernel satisfies the following properties:
\begin{enumerate}[(i)]
    \item $K(z,w)=K(w,z)$ and $K(\lambda z,w)=K(z,\lambda w),\quad \lambda\in\mathbb{C}$.
    \item  For any $\nu\in\mathbb{N}^n$, we have
$$|D_z^\nu K(x,z)|\leq |x|^{|\nu|}e^{|\operatorname{Re}(z)||x|},$$
where $D_z^\nu=\frac{\partial^{|\nu|}}{\partial {z_1}^{\nu_1}\cdots {z_n}^{\nu_n}}$ and $|\nu|=\nu_1+\cdots+\nu_n$. In particular, for any $x,y\in\mathbb{R}^n$, we have
$$|K(x,iy)|\leq 1.$$
\item $K(ix,y)=\overline{K(-ix,y)}$ and $K(gx,gy)=K(x,y),\quad g\in G$.
\end{enumerate}

\subsection{The Dunkl transform} 
For $1\leq p<\infty$, let $L_\kappa^p\left(\mathbb{R}^{n}\right)$ be the space of $L^p$-functions on $\mathbb{R}^{n}$ with respect to the weight $h_\kappa(x)$ with the $L_\kappa^p\left(\mathbb{R}^{n}\right)$-norm
\begin{equation*}
\|f\|_{L_\kappa^p\left(\mathbb{R}^{n}\right)}=\left(\int_{\mathbb{R}^{n}}|f(x)|^ph_\kappa(x)dx\right)^\frac{1}{p},
\end{equation*}
and $\|f\|_{L_\kappa^\infty\left(\mathbb{R}^{n}\right)}=ess \sup\limits_{x\in\mathbb{R}^n}|f(x)|$.
 The Dunkl transform   of  $f\in  L_\kappa^1\left(\mathbb{R}^{n}\right)$ is defined  by
 \begin{equation}\label{Dunkltransform}
 \mathcal{F}_\kappa f(y)=\frac{1}{c_\kappa}\int_{\mathbb{R}^{n}} f(x) K(x,-iy)h_\kappa(x)dx,
 \end{equation}
 where $c_\kappa$ is the Mehta-type constant defined by
 \begin{equation*}
 c_\kappa=\int_{\mathbb{R}^{n}} e^{-\frac{|x|^2}{2}}h_\kappa(x)dx.
 \end{equation*}
When $\kappa \equiv 0$, the Dunkl transform coincides with the classical Fourier transform. The Dunkl transform plays the same role as the Fourier transform in classical Fourier analysis and enjoys properties similar to those of the classical Fourier transform. With $N=2\gamma+n$, here we list some basic properties of the Dunkl transform (see \cite{Dunklll, Jeu, Dai-Ye}).
\begin{enumerate}[(i)]
    \item  For all $f\in L_\kappa^1\left(\mathbb{R}^{n}\right)$, we have
\begin{equation*}
\|\mathcal{F}_\kappa f\|_{L_\kappa^\infty\left(\mathbb{R}^{n}\right)}\leq \frac{1}{c_\kappa} \|f\|_{L_\kappa^1\left(\mathbb{R}^{n}\right)}.
\end{equation*}
\item For any function $f$ in the Schwartz space $\mathcal{S}\left(\mathbb{R}^{n}\right)$, we have
\begin{equation*}
\mathcal{F}_\kappa(T_\xi f)(y)=i\langle \xi, y\rangle \mathcal{F}_\kappa f(y),\quad y,\xi\in\mathbb{R}^{n}.
\end{equation*}
In particular, it follows that 
\begin{equation*}
\mathcal{F}_\kappa(\Delta_\kappa f)(y)=-|y|^2 \mathcal{F}_\kappa f(y),\quad y\in\mathbb{R}^{n}.
\end{equation*}
\item  For all $f\in L_\kappa^1\left(\mathbb{R}^{n}\right)$, we have
\begin{equation*}
\mathcal{F}_\kappa (f(\cdot/\lambda))(y)=\lambda^N\mathcal{F}_\kappa f(\lambda y),\quad y\in\mathbb{R}^{n},
\end{equation*}for all $\lambda>0$.
\item The Dunkl transform $\mathcal{F}_\kappa$ is a homeomorphism of the Schwartz space $\mathcal{S}\left(\mathbb{R}^{n}\right)$ and its inverse is given by $\mathcal{F}_\kappa^{-1}g(x)=\mathcal{F}_\kappa g(-x),$ for all $ g\in \mathcal{S}\left(\mathbb{R}^{n}\right)$. In addition, for all $f\in \mathcal{S}\left(\mathbb{R}^{n}\right)$, it satisfies
\begin{equation*}
 \int_{\mathbb{R}^{n}} |f(x)|^2 h_\kappa(x)dx=\int_{\mathbb{R}^{n}} |\mathcal{F}_\kappa f(y)|^2 h_\kappa(y)dy.
 \end{equation*}
In particular, the Dunkl transform extends to an isometric isomorphism on $L_\kappa^2\left(\mathbb{R}^{n}\right)$. The Dunkl transform can be extended to the space of tempered distributions $\mathcal{S}'\left(\mathbb{R}^{n}\right)$ and is also a homeomorphism of $\mathcal{S}'\left(\mathbb{R}^{n}\right)$.
\item For every $f\in L_\kappa^1\left(\mathbb{R}^{n}\right)$ such that $\mathcal{F}_\kappa f\in L_\kappa^1\left(\mathbb{R}^{n}\right)$, we have the inversion formula
\begin{equation*}
f(x)=\frac{1}{c_\kappa}\int_{\mathbb{R}^{n}} \mathcal{F}_\kappa f(y) K(ix,y)h_\kappa(y)dy,\quad  a.e. \;x\in\mathbb{R}^{n}.
\end{equation*}
\item  If $f$ is a radial function in $ L_\kappa^1\left(\mathbb{R}^{n}\right)$ such that $f(x)=\tilde{f}(|x|)$, then
\begin{equation}\label{radial-Dunkl}
\mathcal{F}_\kappa f(y)=\frac{1}{\Gamma(N/2)}\int_0^\infty \tilde{f}(r)\frac{J_\frac{N-2}{2}(r|y|)}{\left(r|y|\right)^\frac{N-2}{2}}r^{N-1} dr,
\end{equation}
where $J_\nu$ denotes the Bessel function of order $\nu>-\frac{1}{2}$.
\end{enumerate}
 
\subsection{Dunkl convolution operator} For  given $x\in\mathbb{R}^n$, the Dunkl translation operator $f\mapsto \tau_x f$ is defined on $\mathcal{S}\left(\mathbb{R}^{n}\right)$ by
\begin{equation*}
\mathcal{F}_\kappa(\tau_x f)(y)=K(x,iy)\mathcal{F}_\kappa f(y),\quad y\in\mathbb{R}^n.
\end{equation*}
At the moment, the explicit formula of the Dunkl translation operator is known only in two cases. One is when $f(x)=\tilde{f}(|x|)$ is a continuous radial function in $L_k^2(\mathbb{R}^n)$, the Dunkl translation operator is represented by (see \cite{ros, Dai-Wang})
\begin{equation*}
\tau_xf(y)=V_\kappa[\tilde{f}\left(\sqrt{|y|^2+|x|^2-2\langle y,\cdot\rangle}\right)](x)=\int_{\mathbb{R}^n}\tilde{f}\left(\sqrt{|y|^2+|x|^2-2\langle y,\xi\rangle}\right)d\rho^\kappa_x(\xi).
\end{equation*}
The Dunkl translation operator can also be extended to the space of tempered distributions $\mathcal{S}'\left(\mathbb{R}^{n}\right)$. It turns out to be rather difficult to extend important results in classical Fourier analysis to the setting of the Dunkl transform in general. One of the difficulties comes from the fact that the Dunkl translation operator is not positive in
general. In fact, even   $L_\kappa^p(\mathbb{R}^n)$-boundedness of $\tau_x$ is not established in general. On the other hand, when $f$ is a radial function in $L_\kappa^p(\mathbb{R}^n)$, $1\leq p\leq \infty$, the following holds (see \cite{GIT2019, Thangavelu-Xu2005})
\begin{equation*}
\|\tau_xf\|_{L_\kappa^p(\mathbb{R}^n)}\leq \|f\|_{L_\kappa^p(\mathbb{R}^n)}.
\end{equation*}
Using the Dunkl translation operator, we define the Dunkl convolution  of functions $f,g\in \mathcal{S}\left(\mathbb{R}^{n}\right)$ by
\begin{equation*}
f*_\kappa g(x)=\int_{\mathbb{R}^n}\tau_xf(-y)g(y)h_\kappa(y)dy, \quad x\in\mathbb{R}^n.
\end{equation*}
The Dunkl convolution satisfies the following properties (see \cite{Thangavelu-Xu2005}).
\begin{enumerate}
    \item  $\mathcal{F}_\kappa(f*_\kappa g)=\mathcal{F}_\kappa (f)\mathcal{F}_\kappa (g)$ and $f*_\kappa g=g*_\kappa f$.
\item Young's inequality: Let $1\leq p,q,r\leq\infty$ such that $1+\frac{1}{r}=\frac{1}{p}+\frac{1}{q}$. If $f\in L^p_\kappa(\mathbb{R}^n)$ and $g$ is a radial function of $L_\kappa^q(\mathbb{R}^n)$, then $f*_\kappa g\in L_\kappa^r(\mathbb{R}^n)$ and we have
\begin{equation}\label{Young}
\|f*_\kappa g\|_{L_\kappa^r(\mathbb{R}^n)}\leq \|f\|_{L_\kappa^p(\mathbb{R}^n)} \|g\|_{L_\kappa^q(\mathbb{R}^n)}.
\end{equation}
\end{enumerate}
\subsection{Dunkl Besov spaces}\label{HBS}
In this subsection, we recall the definition of the Dunkl Besov spaces using Littlewood-Paley decomposition related to the Dunkl Laplacian $\Delta_\kappa$ (see \cite{Kawazoe-Mejjaoli, Mejjaoli2008}). Dunkl-Besov spaces will be a convenient tool to prove many of our results. 

Fix a non-increasing smooth function $R$ on $[0,+\infty)$ such that $R=1$ in $[0,1/2]$ and $R=0$ in $[1,+\infty)$. Let $\psi(r)=R(r/2)-R(r)$ and $0\leq \psi\leq 1$ such that $\supp \psi\subseteq C_0=[1/2,2]$. Moreover, we observe that
	\begin{align*}
R(r)+\sum_{j=0}^\infty \psi(2^{-j}r)&=1, \quad  r\geq0,\\
\underset{j\in \mathbb{Z}}{\sum}\psi(2^{-j}r)&=1, \quad  r>0.
	\end{align*}
For $u\in \mathcal{S}'\left(\mathbb{R}^{n}\right)$, we define the Littlewood-Paley decomposition related to the Dunkl Laplacian by
$S_0u=R(\sqrt{-\Delta_\kappa})u=\mathcal{F}_\kappa^{-1}\left(R(|\cdot|) \mathcal{F}_\kappa u\right)$ and $\tilde{\Delta}_ju=\psi(2^{-j}\sqrt{-\Delta_\kappa})u=\mathcal{F}_\kappa^{-1}\left(\psi(2^{-j}|\cdot|)\mathcal{F}_\kappa u\right),$ for all $  j\in\mathbb{Z}$. Moreover, we  have
\begin{equation*}
u=S_0u+\sum_{j=0}^\infty \tilde{\Delta}_ju\;\in \;\mathcal{S}'\left(\mathbb{R}^{n}\right).
\end{equation*}
Moreover, we also have   
$$S_0u=S_0(\sum_{j\in \mathbb Z}\tilde{\Delta}_j)u= \sum_{j\leq 0 }S_0\tilde{\Delta}_ju.$$
In addition, from Bernstein's inequalities   \cite{Kawazoe-Mejjaoli}, for any $\sigma\in\mathbb{R}$, $j\in\mathbb{Z}$, $1\leq p,q\leq \infty$ and $u\in \mathcal{S}'\left(\mathbb{R}^{n}\right)$, we have  
	\begin{align}\label{LPLP}
	\left\|(-\Delta_\kappa)^{\frac{\sigma}{2}}\tilde{\Delta}_ju\right\|_{L_\kappa^p(\mathbb{R}^{n})} &\lesssim 2^{j\sigma}\|\tilde{\Delta}_ju\|_{L_\kappa^p(\mathbb{R}^{n})},\\
 \left\|\tilde{\Delta}_ju\right\|_{L_\kappa^q(\mathbb{R}^{n})} &\lesssim 2^{jN\left(\frac{1}{p}-\frac{1}{q}\right)}\|\tilde{\Delta}_ju\|_{L_\kappa^p(\mathbb{R}^{n})}.
	\end{align}
For $1\leq p,q \leq \infty$ and $ s \in\mathbb{R}$, the Dunkl Besov space $B^{s,\kappa}_{p,q}$ is defined as the set of tempered distributions $u \in \mathcal{S}'\left(\mathbb{R}^{n}\right)$ such that $\|u\|_{B^{s,\kappa}_{p,q}}<\infty$ with
	\begin{equation*}
\|u\|_{B^{s,\kappa}_{p,q}}=\|S_0u\|_{L_\kappa^p(\mathbb{R}^n)}+\begin{cases}\left(\sum\limits_{j=0}^\infty2^{jsq}\|\tilde{\Delta}_ju\|_{L_\kappa^p(\mathbb{R}^n)}^q\right)^{\frac{1}{q}},\quad 1\leq q<\infty\\
 \sup_{j\geq0} 2^{js}\|\tilde{\Delta}_ju\|_{L_\kappa^p(\mathbb{R}^n)}\quad,\quad q=\infty.
 \end{cases}
	\end{equation*}
Also, for   $1\leq p\leq \infty$ and $s \in\mathbb{R}$, the Dunkl Sobolev space $H_p^{s,\kappa}(\mathbb{R}^n)$ is
	defined as the set of tempered distributions $u \in \mathcal{S}'\left(\mathbb{R}^{n}\right)$ such that $\mathcal{F}^{-1}_\kappa\left((1+|\cdot|^2)^\frac{s}{2}\mathcal{F}_\kappa u\right)\in L_\kappa^p(\mathbb{R}^n)$ with the norm
 \begin{equation*}
	\|u\|_{H_p^{s,\kappa}(\mathbb{R}^n)}=\|\mathcal{F}^{-1}_\kappa\left((1+|\cdot|^2)^\frac{s}{2}\mathcal{F}_\kappa u\right)\|_{L_\kappa^p(\mathbb{R}^n)}.
	\end{equation*}
Throughout the paper, we  denote $H^{s,\kappa}(\mathbb{R}^n)$  for $H_2^{s,\kappa}(\mathbb{R}^n)$.  Moreover, we have the  equivalent property that 
 \begin{equation*}
	u\in H^{s,\kappa}(\mathbb{R}^n)\Leftrightarrow u\in B^{s,\kappa}_{2,2}(\mathbb{R}^n)
.	\end{equation*}
 Now we list some important properties related to the Besov space $B^{s,\kappa}_{p,q}(\mathbb{R}^n)$ in the following proposition.
\begin{proposition} \cite{Kawazoe-Mejjaoli}\label{properties} Let $p,q\in [1,\infty]$ and $s \in\mathbb{R}$. Then 
\begin{enumerate}
    \item The space $B^{s,\kappa}_{p,q}(\mathbb{R}^n)$ is a Banach space with the norm $||\cdot||_{B^{s,\kappa}_{p,q}(\mathbb{R}^n)}$.
\item The definition of $B^{s,k}_{p,q}(\mathbb{R}^n)$ does not depend on the choice of the function $R$ in the Littlewood-Paley decomposition.
\item The dual space of $B^{s,\kappa}_{p,q}(\mathbb{R}^n)$ is $B^{-s,\kappa}_{p',q'}(\mathbb{R}^n)$;
\item For any $\sigma>0$,  we have the continuous inclusion   $B^{s+\sigma,\kappa}_{p,q}(\mathbb{R}^n)\subseteq B^{s,\kappa}_{p,q}(\mathbb{R}^n)$.
\item Let   $u\in \mathcal{S}'\left(\mathbb{R}^{n}\right)$ and $\sigma\in\mathbb{R}$, then $u\in B^{s,\kappa}_{p,q}(\mathbb{R}^n)$ if and only if $(-\Delta_\kappa)^{\sigma/2}u\in B^{s-\sigma,\kappa}_{p,q}(\mathbb{R}^n)$.
\item For any $1\leq p_1\leq p_2\leq\infty$, $1\leq q_1\leq q_2\leq\infty,$ and $s_1\leq s_2 \in\mathbb{R}$ such that $s_1-\frac{N}{p_1}=s_2-\frac{N}{p_2}$, we also have the  continuous inclusion  
$ B^{s_1,\kappa}_{p_1,q_1}(\mathbb{R}^n)\subseteq B^{s_2,\kappa}_{p_2,q_2}(\mathbb{R}^n).
$ 
\item For all $p\in[2, \infty]$, we have the continuous inclusion $B^{0,\kappa}_{p,2}(\mathbb{R}^n)\subseteq L_\kappa^p(\mathbb{R}^n)$.
\item $B^{0,\kappa}_{2,2}(\mathbb{R}^n)=L_\kappa^2(\mathbb{R}^n)$.
\end{enumerate}
\end{proposition}

\subsection{Technical Lemmas} 
In this subsection, we recall several important properties to prove the main result of this paper.  From the Dunkl transform and Young's inequality, the $L_\kappa^1-L_\kappa^\infty$ decay estimates are reduced to a list of oscillatory integrals related to Bessel functions. In order to estimate the oscillatory integrals, we recall the stationary phase lemmas and the properties of Bessel functions. We first recall the van der Corput lemma.
\begin{lemma}[Van der Corput lemma \cite{S1993}]\label{phase} Let $g\in C^\infty([a,b])$ be real-valued such that $|g''(x)|\geq \delta $ for any $x\in[a,b]$ with $\delta >0$. Then for any function $\psi \in C^\infty([a,b])$, there exists a constant $C$ (does not depend on $\delta, a, b, g$ or $\psi$)  such that
	\begin{equation*}
	\left|\int_a^b e^{{ ig(x)}}\psi(x)\,dx\right|\leq C{ \delta ^{-1/2}}\left(\|\psi\|_\infty+\|\psi'\|_1\right).
	\end{equation*}
 \end{lemma}
Let $J_\nu$ be the Bessel function of order $\nu>-\frac{1}{2}$, defined as
\begin{equation*}
J_\nu(r)=\frac{(\frac{r}{2})^\nu}{\Gamma(\nu+\frac{1}{2})\pi^{\frac{1}{2}}}\int_{-1}^1e^{ir\tau}(1-\tau^2)^{\nu-\frac{1}{2}}d\tau.
\end{equation*}
We list some properties of $J_\nu$ in the following lemma.
\begin{lemma}[see \cite{Grafakos2008}]\label{Bessel}
For $r>0$ and $\nu>-\frac{1}{2}$, we have
\begin{align}
&(1)J_\nu(r)\leq C_\nu r^\nu,\label{bessel1}\\
&(2)\frac{d}{dr}\left(r^{-\nu}J_\nu(r)\right)=-r^{-\nu}J_{\nu+1}(r),\label{bessel2}\\
&(3)J_\nu(r)\leq C_\nu r^{-\frac{1}{2}}, \label{bessel3}\\
&(4)J_\nu(r)=\frac{(\frac{r}{2})^\nu}{\Gamma(\nu+\frac{1}{2})\Gamma(\frac{1}{2})} \left[ie^{-ir}\int_0^\infty e^{-rt}(t^2+2it)^{\nu-\frac{1}{2}}dt-ie^{ir}\int_0^\infty e^{-rt}(t^2-2it)^{\nu-\frac{1}{2}}dt\right].\label{bessel4}
\end{align}
\end{lemma}
\begin{remark}
From \eqref{bessel1} and \eqref{bessel2} of Lemma \ref{Bessel}, we can easily obtain that for any $0\leq s\leq 2$ and $\beta\in\mathbb{N}_{0}$,
\begin{equation}\label{small-s}
  \left|\frac{d^\beta}{dr^\beta}\left(\psi(r)\frac{J_{\frac{N-2}{2}}(rs)}{(rs)^\frac{N-2}{2}}  r^{N-1}\right)\right|\leq C_{\beta},
\end{equation}
where $\psi$ is the function in Section \ref{HBS}.
Again from \eqref{bessel4}, we have the identity
\begin{equation}\label{Bessel-Fourier}
    \frac{J_{\frac{N-2}{2}}(r)}{r^\frac{N-2}{2}}=C\left(e^{ir}h(r)+e^{-ir}\overline{h(r)}\right),
\end{equation}
 where $$h(r)=-i\int_0^\infty e^{-rt}(t^2-2it)^\frac{N-3}{2}dt.$$ Moreover,  integrating by parts (if necessary) and using Van der Corput’s Lemma,  for any $\beta\in\mathbb{N}_{0}$, one can get     
\begin{equation}\label{h}
\left|\frac{d^\beta}{dr^\beta}h(r)\right|\leq C_\beta(1+r)^{-\frac{N-1}{2}-\beta},
\end{equation}
for all $\beta\in\mathbb{N}_{0}$. Thus, from \eqref{h}, for any $s>2$ and $\beta\in\mathbb{N}_{0}$, we have
\begin{equation}\label{big-s}
\left|\frac{d^\beta}{dr^\beta}\left(\psi(r)h(rs) r^{N-1}\right)\right|\leq C_\beta s^{-\frac{N-1}{2}}.
\end{equation}
\end{remark}
We also exploit the following estimates, which can be easily proved by comparing the sums with the corresponding integrals.
\begin{lemma}\label{Sum}
Fix $\beta>0$. There exists $C_\beta>0$ such that for any $A>0$, we have
\begin{align*}
\sum_{j\in\mathbb{Z}, 2^j\leq A}2^{j\beta}&\leq C_\beta A^\beta,\\
\sum_{j\in\mathbb{Z}, 2^j>A}2^{-j\beta}&\leq C_\beta A^{-\beta}.
\end{align*}
\end{lemma}
Finally, we apply the following duality arguments (see \cite{GV}) to prove the Strichartz estimates for some concrete equations related to the Dunkl Laplacian.
\begin{lemma}\label{equ} Let $H$ be a Hilbert space, $X$ be a Banach space, $X^*$ is the dual of $X$, and $D$ be a vector space densely contained in $X$. Let $A\in \mathcal{L}_a(D,H)$ and   $A^*\in \mathcal{L}_a(H,D_a^*)$ be its adjoint operator, defined by
	\begin{equation*}
	\langle A^*v,f \rangle_D=\langle v,Af \rangle_H, \quad \forall f\in D, \quad \forall v \in H,
	\end{equation*}
where $\mathcal{L}_a(Y,Z)$ is the space of linear maps from a vector space $Y$ to a vector space $Z$, $D^*_a$ is the algebraic dual of $D$, $\langle \varphi,f\rangle_D$ is the pairing between $D^*_a$ and $D$, and $\langle\cdot,\cdot\rangle_H$ is the scalar product in $H$. Then the following three conditions are equivalent:
\begin{enumerate}
    \item There exists  $0\leqslant a \leqslant \infty$ such that for all $f \in D$,
	\begin{equation*}
	\|Af\| \leqslant a\|f\|_X.
	\end{equation*}
 \item  $\mathfrak{R}(A^*)\subset X^*$ and there exists   $0\leqslant a \leqslant \infty$, such that for all $v \in H$,
	\begin{equation*}
	\|A^*v\|_{X^*} \leqslant a\|v\|.
	\end{equation*}
 \item $\mathfrak{R}(A^*A)\subset X^*$,and there exists $a$, $0\leqslant a \leqslant \infty$, such that for all $f \in D$,
	\begin{equation*}
	\|A^*Af\|_{X^*} \leqslant a^2\|f\|_X,
	\end{equation*}
where $\|\cdot\|$ denote the norm in $H$. 
\end{enumerate}
  The constant $a$ is the same in all three parts. If one of those conditions is satisfied, the operators $A$ and $A^*A$ extend by continuity to bounded operators from $X$ to $H$ and from $X$ to $X^*$, respectively.
\end{lemma}
\begin{lemma} \label{DuilatyXX} Let $H$, $D$ and two triplets $(X_i,A_i,a_i), i=1,2$, satisfy any of the conditions of Lemma \ref{equ}. Then for all choices of $i,j=1,2$, $\mathfrak{R}(A_i^*A_j) \subset X_i^*$, and for all $f \in D$, we have
	\begin{equation*}
	\|A_i^* A_jf\|_{X_i^*} \leqslant a_i a_j\|f\|_{X_j}.
	\end{equation*}
\end{lemma}
\begin{lemma}\label{L1LW}  Let $H$ be a Hilbert space, $I$ be an interval of $\mathbb{R}$,    $X \subset \mathcal{S}'(I \times \mathbb{R}^n)$ be a Banach space which is stable under time restriction, and let $X$ and $A$ satisfy (any of) the conditions of Lemma \ref{equ}. Then the operator $A^*A$ is a bounded operator from $L_t^1(I,H)$ to $X^*$ and from $X$ to $L_t^\infty(I,H)$.
\end{lemma}
\section{Proof of main results}\label{sec4}
This section is devoted to proving  Theorem \ref{ResultTime}, i.e.,  decay estimate for a class of
dispersive   semigroup $e^{it\phi(\sqrt{-\Delta_\kappa})}$ on $\mathbb{R}^n$.
 \begin{proof}[Proof of Theorem \ref{ResultTime}]
 In order to prove  part \textbf{(A)} first,   we note that \begin{equation}\label{Ij}
U_t\tilde{\Delta}_ju_0=\mathcal{F}_\kappa^{-1}\left(e^{it\phi(|\cdot|)}\psi(2^{-j}|\cdot|)\mathcal{F}_\kappa u_0\right)=\mathcal{F}_\kappa^{-1}\left(e^{it\phi(|\cdot|)}\psi(2^{-j}|\cdot|)\right)*_\kappa u_0.
\end{equation}
It follows from Young's inequality \eqref{Young} that 
\begin{equation*}
\|U_t\tilde{\Delta}_ju_0\|_{L_\kappa^\infty(\mathbb{R}^n)}\leq \|\mathcal{F}_\kappa^{-1}\left(e^{it\phi(|\cdot|)}\psi(2^{-j}|\cdot|)\right)\|_{L_\kappa^\infty(\mathbb{R}^n)}\|u_0\|_{L_\kappa^1(\mathbb{R}^n)}.
\end{equation*}
 Now our aim is to  estimate $I_j(x)=\mathcal{F}_\kappa^{-1}\left(e^{it\phi(|\cdot|)}\psi(2^{-j}|\cdot|)\right)(x)$. From the Dunkl transform formula of a radial function \eqref{radial-Dunkl}, we have
\begin{equation}\label{IIj}
\begin{aligned}
I_j(x)&=\mathcal{F}_\kappa^{-1}\left(e^{it\phi(|\cdot|)}\psi(2^{-j}|\cdot|)\right)(x)\\
&=\frac{1}{\Gamma(N/2)}\int_0^\infty e^{it\phi(r)}\psi(2^{-j}r)\frac{J_\frac{N-2}{2}(r|x|)}{\left(r|x|\right)^\frac{N-2}{2}}r^{N-1} dr\\
&=\frac{{ 2^{jN}}}{\Gamma(N/2)}\int_0^\infty e^{it\phi(2^jr)}\psi(r)\frac{J_\frac{N-2}{2}(2^jr|x|)}{\left(2^jr|x|\right)^\frac{N-2}{2}}r^{N-1} dr\\
&=\uppercase\expandafter{\romannumeral2}_{j}(2^js),
\end{aligned}
\end{equation}
where $s=|x|$ and 
\begin{equation*}
\uppercase\expandafter{\romannumeral2}_{j}(s)=\frac{ 2^{jN}}{\Gamma(N/2)}\int_0^\infty e^{it\phi(2^jr)}\psi(r)\frac{J_\frac{N-2}{2}(rs)}{\left(rs\right)^\frac{N-2}{2}}r^{N-1} dr.
\end{equation*}
We will continue the discussion in two cases. First, we use the vanishing property at the origin \eqref{bessel1} and the recurring property for the Bessel functions \eqref{bessel2}. On the other hand, we apply the decay property of the Bessel function \eqref{Bessel-Fourier}.

\textbf{\underline{Case 1.} $s\leq 2$:} 
By denoting $D_r=\frac{1}{i2^{j}t\phi'(2^{j}r)}\frac{d}{dr}$, we notice  that
\begin{equation*}
D_r\left(e^{it\phi(2^{j}r)}\right)=e^{it\phi(2^{j}r)}, \; D_r^*f=\frac{i}{2^{j}t}\frac{d}{dr}\left(\frac{1}{\phi'(2^{j}r)}f\right).
\end{equation*}
For $\beta\in\mathbb{N}$ and $r\in [1/2,2]$, it follows from the hypothesis (C1) that
\begin{equation}\label{phase-prime}
 \frac{d^\beta}{dr^\beta}\left(\frac{1}{\phi'(2^{j}r)}\right)\leq C_\beta 2^{-j(m_1-1)}.
\end{equation}
Using integration by part, we have
 \begin{equation*}
	\begin{aligned}
	\uppercase\expandafter{\romannumeral2}_j(s)&=\frac{{ 2^{jN}}}{\Gamma(N/2)}\int_0^\infty e^{it\phi(2^jr)}\psi(r)\frac{J_\frac{N-2}{2}(rs)}{\left(rs\right)^\frac{N-2}{2}}r^{N-1} dr\\
    &=\frac{{ 2^{jN}}}{\Gamma(N/2)} \int_0^\infty D_r\left(e^{it\phi(2^{2j}r)}\right)\psi(r)\frac{J_\frac{N-2}{2}(rs)}{\left(rs\right)^\frac{N-2}{2}}r^{N-1} dr\\
    &=\frac{{ 2^{jN}}}{\Gamma(N/2)}\cdot \frac{i}{2^{j}t}\int_0^\infty e^{it\phi(2^jr)}\frac{d}{dr}\left(\frac{1}{\phi'(2^{j}r)}\psi(r)\frac{J_\frac{N-2}{2}(rs)}{\left(rs\right)^\frac{N-2}{2}}r^{N-1}\right) dr\\
    &=\frac{{ 2^{jN}}}{\Gamma(N/2)}\cdot \left(\frac{i}{2^{j}t}\right)^{ Q}\sum_{\ell=0}^Q\sum_{(\beta_1,\beta_2,\cdots,\beta_Q)\in \Lambda_Q^\ell} C_{Q,\ell}\int_0^\infty  e^{it\phi(2^{j}r)}\\
     &\qquad \qquad\times \prod_{i=1}^Q\frac{d^{\beta_i}}{dr^{\beta_i}}\left(\frac{1}{\phi'(2^{j}r)}\right)\frac{d^{Q-\ell}}{dr^{Q-\ell}}\left(\psi(r)\frac{J_\frac{N-2}{2}(rs)}{\left(rs\right)^\frac{N-2}{2}}r^{N-1} \right)\,dr,
    \end{aligned}
	\end{equation*}
for any $Q\in \mathbb N$,  where $\Lambda_Q^\ell=\{(\beta_1,\beta_2,\cdots,\beta_Q)\in\{0,1,2,\cdots,\ell\}^Q: \sum\limits_{i=1}^Q\beta_i=\ell\}$. Now using   the recurring property for the Bessel functions \eqref{bessel2}, the vanishing property at the origin \eqref{bessel1}, and the estimate \eqref{phase-prime}, we have
 \begin{equation}\label{sharp-s-small}
|\uppercase\expandafter{\romannumeral2}_j(s)|\leq C_{Q,N,k}|t|^{-Q}2^{j(N-m_1Q)}, 
 \end{equation} for any $Q\in \mathbb N$. Using Riesz-Thorin interpolation   between different $Q\in\mathbb{N}$ in \eqref{sharp-s-small}, for any $  \theta\geq0$, we have
\begin{equation}\label{sharp-s-small}
|\uppercase\expandafter{\romannumeral2}_j(s)|\lesssim |t|^{-\theta}2^{j(N-m_1\theta)}.
 \end{equation}  
 This completes the proof of (A) in this case.   
 
\textbf{\underline{Case 2.}} $s>2$. In this case, we shall apply the decay property of Bessel functions.
From  the identity \eqref{Bessel-Fourier}, we can write
\begin{align*}
\uppercase\expandafter{\romannumeral2}_j(s)
&=C2^{jN}\int_0^\infty  e^{it\phi(2^{j}r)}\psi(r)\left(e^{irs}h(rs)+e^{-irs}\overline{h(rs)}\right)r^{N-1} \,dr\\
    &=C2^{jN}\int_0^\infty  e^{i\left(t\phi(2^{j}r)+rs\right)}\psi(r) h(rs) r^{N-1}
    \,dr\\
    &\qquad +C2^{jN}\int_0^\infty  e^{i\left(t\phi(2^{j}r)-rs\right)}\psi(r) \overline{h(rs)} r^{N-1}
    \,dr\\ 
    &= \uppercase\expandafter{\romannumeral2}^{+}_j(s)+\uppercase\expandafter{\romannumeral2}^{-}_j(s),
\end{align*}
where
\begin{align*}
\uppercase\expandafter{\romannumeral2}^+_j(s)&=C2^{jN}\int_0^\infty  e^{i\left(t\phi(2^{j}r)+rs\right)}\psi(r) h(rs) r^{N-1}
    \,dr=C2^{jN}\int_0^\infty  e^{it\phi_+(2^jr)}\psi(r) h(rs) r^{N-1}
    \,dr,\\
\uppercase\expandafter{\romannumeral2}^-_j(s)&=C2^{jN}\int_0^\infty  e^{i\left(t\phi(2^{j}r)-rs\right)}\psi(r) \overline{h(r|s|)} r^{N-1}
    \,dr=C2^{jN}\int_0^\infty  e^{it\phi_-(2^jr)}\psi(r) \overline{h(rs)} r^{N-1}
    \,dr,
\end{align*}
with $\phi_\pm(r)=\phi(r)\pm \frac{rs}{2^{j}t}$.
Without loss of generality, we can assume that $t>0$ and $\phi'(r)>0$. Now we will estimate $\uppercase\expandafter{\romannumeral2}^+_j(s)$ and $\uppercase\expandafter{\romannumeral2}^{-}_j(s)$ separately.

 For $\uppercase\expandafter{\romannumeral2}^+_j(s)$, note that $\phi'_+(r)\geq \phi'(r)$ and \eqref{phase-prime} also holds true if we replace $\phi$ by $\phi_+$. Then, proceeding similarly as \textbf{Case 1}, combined with \eqref{big-s},    for any $\theta\geq0,$ we obtain
\begin{equation}\label{sharp-s-big-plus}
|\uppercase\expandafter{\romannumeral2}^+_j(s)|\lesssim |t|^{-\theta}2^{j(N-m_1\theta)}.
 \end{equation}
 
 For $\uppercase\expandafter{\romannumeral2}^-_j(s)$, $\phi'_-(2^{j}r)=\phi'(2^{j}r)-\frac{s}{2^{j}t}$. Noticing that if $s=2^{j}t\phi'(2^{j}r)$, then $\phi'_-(2^{j}r)=0$. We continue our discussion by considering the following three sub-cases:
\begin{itemize}
\item When  $s\geq 2\sup\limits_{r\in[1/2,2]}2^{j}t\phi'(2^{j}r)$:\\
In this case, we have $|\phi'_-(2^{j}r)|=\frac{1}{2^{j}t}\left(s-2^{j}t\phi'(2^{j}r)\right)\geq \phi'(2^{j}r)$ and also the estimate 
 \eqref{phase-prime}   hold true if we replace $\phi$ by $\phi_-$. Then, analogous to $\uppercase\expandafter{\romannumeral2}^+_{j}(s)$, for any $\theta\geq0$, we obtain
\begin{equation}\label{sharp-s-big-minus1}
|\uppercase\expandafter{\romannumeral2}^-_j(s)|\lesssim |t|^{-\theta}2^{j(N-m_1\theta)}.
 \end{equation}

\item  When $s\leq \frac{1}{2}\inf\limits_{r\in[1/2,2]}2^{j}t\phi'(2^{j}r)$: \\
 In this case, we have $\phi'_-(2^{j}r)=\frac{1}{2^{j}t}\left(2^{j}t\phi'(2^{j}r)-s\right)\geq \frac{1}{2}\phi'(2^{j}r)$ and \eqref{phase-prime} also hold true if we replace $\phi$ by $\phi_-$. Analogous to $\uppercase\expandafter{\romannumeral2}^+_{j}(s)$, for any $\theta\geq0$, we obtain
\begin{equation}\label{sharp-s-big-minus2}
|\uppercase\expandafter{\romannumeral2}^-_j(s)|\lesssim |t|^{-\theta}2^{j(N-m_1\theta)}.
 \end{equation}
 
\item When $\frac{1}{2}\inf\limits_{r\in[1/2,2]}2^{j}t\phi'(2^{j}r)\leq s\leq 2\sup\limits_{r\in[1/2,2]}2^{j}t\phi'(2^{j}r)$: \\
In this case, we see that $s\sim t2^{jm_1}$. Moreover,  from \eqref{big-s}, it follows  that
 \begin{equation}\label{sharp-s-big-minus3}
 \begin{aligned}
  |\uppercase\expandafter{\romannumeral2}^-_j(s)|&\lesssim 2^{jN}s^{-(N-1)/2}\\
  &\lesssim|t|^{-(N-1)/2}2^{j\left(N-m_1(N-1)/2\right)}.
  \end{aligned}
 \end{equation}
\end{itemize}
Now using the estimate \eqref{h}   in $\uppercase\expandafter{\romannumeral2}^-_j(s)$, a simple calculation yields the following easy estimate
\begin{equation}\label{sharp-s-big-minus3-trival}
  |\uppercase\expandafter{\romannumeral2}^-_j(s)|\lesssim 2^{jN}.
\end{equation}	
Applying interpolation  between \eqref{sharp-s-big-minus3} and  \eqref{sharp-s-big-minus3-trival}, we obtain 
\begin{equation}\label{sharp}
|\uppercase\expandafter{\romannumeral2}^-_j(s)|\lesssim |t|^{-\theta}2^{j\left(N-m_1\theta\right)}
 \end{equation}
 for all $0\leq\theta\leq(N-1)/2$.
Combining  \eqref{sharp} and \eqref{sharp-s-small}, we complete the proof of \eqref{res3-2} in $(A)$.

If (C3) holds in addition, we have $\left|\frac{d^2}{dr^2}\left(\phi_-(2^{j}r)\right)\right|=2^{2j}|\phi''(2^{j}r)|\sim 2^{j\alpha_1}$  in the last sub-case discussed above. It follows from Lemma \ref{phase} and \eqref{big-s} that
\begin{equation}\label{better-sharp-s-big-minus3}
\begin{aligned}
|\uppercase\expandafter{\romannumeral2}^-_j(s)| \lesssim 2^{jN}|t2^{j\alpha_1}|^{-1/2}s^{-(N-1)/2}
\lesssim |t|^{-N/2}2^{j\left(N-m_1 N/2-(\alpha_1-m_1)/2\right)}.
\end{aligned}
 \end{equation}
For any $0\leq \theta\leq 1$,   since $(N-1+\theta)/2=(1-\theta)(N-1)/2+\theta N/2$, by an interpolation between \eqref{sharp} (with $\theta=(N-1)/2$) and \eqref{better-sharp-s-big-minus3}, we obtain
\begin{equation}\label{Better-sharp-s-big-minus3}
|\uppercase\expandafter{\romannumeral2}^-_j(s)|\lesssim |t|^{-(N-1+\theta)/2}2^{j\left(N-m_1(N-1+\theta)/2-\theta(\alpha_1-m_1)/2\right)}.
 \end{equation}
Noting that $\alpha_1\leq m_1$, from \eqref{sharp-s-small}, \eqref{sharp-s-big-plus}, \eqref{sharp-s-big-minus1}, \eqref{sharp-s-big-minus2} and \eqref{Better-sharp-s-big-minus3}, we have
\begin{equation}\label{better-sharp}
|\uppercase\expandafter{\romannumeral2}_j(s)|\lesssim |t|^{-(N-1+\theta)/2}2^{j\left(N-m_1(N-1+\theta)/2-\theta(\alpha_1-m_1)/2\right)},
 \end{equation}
 for all $0\leq\theta\leq 1$. This completes the proof of \eqref{res3-3} in part (A).\\

Now we will prove part {(C)}, as the proof of part {(B)} is similar to part   {(A)}. To prove part  {(C)}, take any fixed $\theta$ satisfying $0\leq\theta\leq\min\left(\frac{N}{m_2},\frac{N-1}{2}\right)$. First, we consider the case  when  $\theta=\frac{N-1}{2}<\frac{N}{m_2}$, i.e., $N-m_2\theta>0$. Now applying part (B), we obtain
 \begin{equation*}
\begin{aligned}
\|U_t S_0u_0\|_{L_\kappa^\infty(\mathbb{R}^n)}&\lesssim\left\|\sum_{j\leq 0}U_t\tilde{\Delta}_jS_0u_0\right\|_{L_\kappa^\infty(\mathbb{R}^n)}\\
&\leq \sum_{j\leq 0}\left\|U_t\tilde{\Delta}_jS_0u_0\right\|_{L_\kappa^\infty(\mathbb{R}^n)}\\
&\lesssim\sum_{j\leq  0}|t|^{-\theta}2^{j(N-m_2\theta)}\left\|S_0u_0\right\|_{L_\kappa^1(\mathbb{R}^n)}\\
&\lesssim |t|^{-\theta}  \left\|S_0u_0\right\|_{L_\kappa^1(\mathbb{R}^n)}\lesssim  |t|^{-\frac{(N-1)}{2}}\|u_0\|_{L_\kappa^1(\mathbb{R}^n)}.
\end{aligned}
\end{equation*}
Therefore, it suffices to consider the case $\theta=\frac{N}{m_2}\leq\frac{N-1}{2}$. By \eqref{Ij} and \eqref{IIj}, we have
\begin{equation}\label{U2}
U_t S_0u_0=\sum_{j\leq 0}U_t\tilde{\Delta}_j (S_0u_0)=\left(\sum_{j\leq 0} \uppercase\expandafter{\romannumeral2}_j\left(2^j|\cdot|\right)\right)*_\kappa(S_0u_0).
\end{equation}
Argued similarly as (a) but for the proof of (b), we know that if $j_0< 0$ and $  s\sim |t|2^{j_0m_2}>2$, then
\begin{equation}\label{sim}
 \begin{aligned}
  \left|\uppercase\expandafter{\romannumeral2}_{j_0}(2^{j_0}s)\right|
  &\lesssim |t|^{-(N-1)/2}2^{j_0\left(N-m_2(N-1)/2\right)}\\
  &\lesssim |t|^{-(N-1)/2}\left|\frac{s}{t}\right|^{\left(N-m_2(N-1)/2\right)\frac{1}{m_2}}\\
  &\lesssim |t|^{-\frac{N}{m_2}}.
  \end{aligned}
 \end{equation}
When $|j-j_0|>C\gg1$, it holds true that
 \begin{equation}\label{gg}
\left|\uppercase\expandafter{\romannumeral2}_{j}(2^js)\right|\lesssim|t|^{-\alpha}2^{j(N-m_2\alpha)},\;\forall \alpha\geq0.
 \end{equation}
Hence, by \eqref{sim} and \eqref{gg}, we get
 \begin{align}
  &\quad \sum_{j\leq0}\left|\uppercase\expandafter{\romannumeral2}_j\left(2^js\right)\right|\nonumber\\
  &\leq \left(\sum_{|j-j_0|\leq C}+\sum_{|j-j_0|>C,\;2^j\leq|t|^{-\frac{1}{m_2}}}+\sum_{|j-j_0|>C,\;2^j>|t|^{-\frac{1}{m_2}}}\right)\left|\uppercase\expandafter{\romannumeral2}_j\left(2^js\right)\right|\nonumber\\
  &\lesssim |t|^{-\frac{N}{m_2}}+\sum_{2^j\leq|t|^{-\frac{1}{m_2}}}2^{jN}+\sum_{2^j>|t|^{-\frac{1}{m_2}}}|t|^{-\alpha}2^{j(N-m_2\alpha)}\nonumber\\
  &\lesssim |t|^{-\frac{N}{m_2}},\label{last}
  \end{align}
 where the last inequality is obtained by applying Lemma \ref{Sum} and choosing $\alpha>\frac{N}{m_2}$.  By Young's inequality \eqref{Young}, combined with \eqref{U2} and \eqref{last}, we obtain
 \begin{align*}
\|U_t S_0u_0\|_{L_\kappa^\infty(\mathbb{R}^n)}&\leq \left\|\sum_{j\leq 0}\uppercase\expandafter{\romannumeral2}_j\left(2^j|\cdot|\right)\right\|_{L_\kappa^\infty(\mathbb{R}^n)}\|S_0u_0\|_{L_\kappa^1(\mathbb{R}^n)}\\
&\lesssim |t|^{-\frac{N}{m_2}}\|S_0u_0\|_{L_\kappa^1(\mathbb{R}^n)}\lesssim  |t|^{-\frac{N}{m_2}}\|u_0\|_{L_\kappa^1(\mathbb{R}^n)},
\end{align*}
which proves \eqref{res-sum1}. In addition, if (C4) holds, we can argue similarly as previously.
 \end{proof}
\section{Applications}\label{sec5}
In this section, we provide several applications of Theorem \ref{ResultTime} by considering some particular smooth function $\phi$. In particular, we prove Strichartz estimates for some concrete wave equations related to the Dunkl Laplacian by using Theorem \ref{ResultTime}. Above all, we introduce the general Strichartz estimates for $U_t=e^{it\phi(\sqrt{-\Delta_k})}$. We start this section by recalling the following versions of the Hardy-Littlewood-Sobolev inequality:
\begin{lemma}[\cite{GPW2008}]\label{VHLW} Assume $\theta_1,\theta_2\in\mathbb{R}$. Let
\begin{equation*}
    k(y)=\begin{cases}
	|y|^{-\theta_1} , \, &|y|\leq1,\\
	|y|^{-\theta_2}, \,&|y|>1.
 \end{cases}
\end{equation*}
Assume that one of the following condition holds:
\begin{enumerate}
    \item $0<\theta_1=\theta_2<n$, $1<p<q<\infty$ and $1-\frac{1}{p}+\frac{1}{q}=\frac{\theta_1}{n}$,
    \item $\theta_1<\theta_2$, $0<\theta_1<n$, $1<p<q<\infty$ and $1-\frac{1}{p}+\frac{1}{q}=\frac{\theta_1}{n}$,
    \item $\theta_1<\theta_2$, $0<\theta_2<n$, $1<p<q<\infty$ and $1-\frac{1}{p}+\frac{1}{q}=\frac{\theta_2}{n}$,
    \item $\theta_1<\theta_2$, $1\leq p\leq q\leq \infty$ and $\frac{\theta_1}{n}<1-\frac{1}{p}+\frac{1}{q}<\frac{\theta_2}{n}$.
\end{enumerate}
We have
\begin{equation*}
\left\|\int_{\mathbb{R}^n}f(\cdot-y)k(y)dy\right\|_{L^q(\mathbb{R}^n)}\lesssim \|f\|_{L^p(\mathbb{R}^n)}.
\end{equation*}
\end{lemma}

\begin{definition}
For $\theta_1, \theta_2\in\mathbb{R}$ with $\theta_1<\theta_2$, we say $r\in E(\theta_1,\theta_2)$ if one of the following conditions holds:
\begin{enumerate}
    \item $0<\theta_1=\theta_2<1$ and $r=\frac{2}{\theta_1}$,
    \item $\theta_1<\theta_2$, $0<\theta_1<1$ and $r=\frac{2}{\theta_1}$,
    \item $\theta_1<\theta_2$, $0<\theta_2<1$ and $r=\frac{2}{\theta_2}$,
    \item $\theta_1<\theta_2$, $2\leq r\leq\infty$ and $\theta_1<\frac{2}{r}<\theta_2$.
\end{enumerate}
\end{definition}
 For $2\leq p\leq \infty$, $s=s(p)\in\mathbb{R}$ and $\theta_1\leq \theta_2$,  we assume that   
\begin{equation}\label{General-Assumption}
    \|U_tf\|_{B_{p,2}^{s,\kappa}(\mathbb{R}^n)}\lesssim k(t) \|f\|_{B_{p',2}^{-s,\kappa}(\mathbb{R}^n)}, \quad\text{where } k(t)=\begin{cases}
	|t|^{-\theta_1} , \, &|t|\leq1,\\
	|t|^{-\theta_2}, \,&|t|>1.
	\end{cases}
\end{equation}
Once we obtain this dispersive estimate with different decay rates between $|t|>1$ and $|t|\leq 1$, we can prove the following proposition by using Lemma \ref{VHLW} and standard duality arguments in Lemma \ref{equ}-\ref{L1LW}.
\begin{proposition}[see \cite{GPW2008, KT}]\label{General-Strichartz}
    Assume $U_t$ satisfies \eqref{General-Assumption}. Then for $r\in E(\theta_1,\theta_2)$ and $T>0$, we have
    \begin{align*}
        \|U_tf\|_{L^r\left(\mathbb{R},B^{s,\kappa}_{p,2}(\mathbb{R}^n)\right)}&\lesssim \|f\|_{L_\kappa^2(\mathbb{R}^n)},\\
        \|\int_0^tU_{t-\tau}g(\tau,\cdot)d\tau\|_{L^r\left([-T,T],B^{s,\kappa}_{p,2}(\mathbb{R}^n)\right)}&\lesssim \|g\|_{L^{r'}\left([-T,T],B^{-s,\kappa}_{p',2}(\mathbb{R}^n)\right)},\\
       \|\int_0^tU_{t-\tau}g(\tau,\cdot)d\tau \|_{L^\infty\left([-T,T],H^{s,\kappa}(\mathbb{R}^n)\right)}&\lesssim \|g\|_{L^{r'}\left([-T,T],B_{p',2}^{-s,\kappa}(\mathbb{R}^n)\right)},\\
    \|\int_0^tU_{t-\tau}g(\tau,\cdot)d\tau\|_{L^r\left([-T,T],B_{p,2}^{s,\kappa}(\mathbb{R}^n)\right)}&\lesssim \|g\|_{L^1\left([-T,T],L_\kappa^2(\mathbb{R}^n)\right)}.
    \end{align*}
\end{proposition}
Note that the endpoint case $\theta_i=1$ also holds true and can be seen from Keel-Tao \cite{KT}.
 
\subsection{Wave equation}
First, we consider the Dunkl-wave equation as
\begin{equation}\label{Wave-equ}
	\begin{cases}
	\partial_t^2u-\Delta_\kappa u=g,\\
	u|_{t=0}=u_0 ,  \\
	\partial_tu|_{t=0}=u_1.
	\end{cases}
	\end{equation}
By Duhamel's principle, the solution is formally given by
\begin{equation*}
	u(t)=\frac{d{\mathcal W}_t}{dt}u_0+{\mathcal W}_tu_1-\int_0^t {\mathcal W}_{t-\tau}g(\tau,\cdot)\,d\tau,
	\end{equation*}
where
	\begin{equation*}
	\mathcal{W}_t=\frac{\sin(t\sqrt{-\Delta_\kappa})}{\sqrt{-\Delta_\kappa}},
	\quad  \frac{d \mathcal{W}_t}{dt}=\cos(t\sqrt{-\Delta_\kappa}).
	\end{equation*}
This reduces to the semigroup $W_t=e^{it\sqrt{-\Delta_\kappa}}$, which corresponds to the case   $\phi(r)=r$. Thus   $\phi$   satisfies (C1) and (C2) with $m_1=m_2=1$.  Using  Theorem \ref{ResultTime}, we obtain the  following decay estimate for the wave operator.
\begin{proposition}
 Assume $2\leq p\leq\infty$, $1\leq q\leq\infty$, $\delta=\frac{1}{2}-\frac{1}{p}$ and $2s\leq -(N+1)\delta$, then we have
\begin{equation*}
    \|W_tf\|_{B_{p,q}^{s,\kappa}(\mathbb{R}^n)}\lesssim k(t)\|f\|_{B_{p',q}^{-s,\kappa}(\mathbb{R}^n)}, \quad k(t)=\begin{cases}
	|t|^{-2\max (s+N\delta,0)}, \, &|t|\leq1,\\
	|t|^{-(N-1)\delta}, \,&|t|\geq1.
	\end{cases}
\end{equation*}
\end{proposition}
\begin{proof}
First, we prove the case $|t|\geq 1$. It follows from \eqref{res-sum1} of Theorem \ref{ResultTime} by taking $\theta=\frac{N-1}{2}$ and Plancherel's identity that
\begin{align*}
 \|W_t S_0f\|_{L_\kappa^\infty (\mathbb{R}^n)}&\lesssim |t|^{-\frac{N-1}{2}} \|S_0f\|_{L_\kappa^1(\mathbb{R}^n)},\\
 \|W_t S_0f\|_{L_\kappa^2 (\mathbb{R}^n)}&\leq \|S_0f\|_{L_\kappa^2(\mathbb{R}^n)}.
	\end{align*}
 From \eqref{res3-2} of Theorem \ref{ResultTime} by taking $\theta=\frac{N-1}{2}$ and Plancherel's identity,  for $j\geq 0$, we have
 \begin{align*}
\left\|W_t\tilde{\Delta}_jf\right\|_{L_\kappa^\infty(\mathbb{R}^n)} &\lesssim |t|^{-\frac{N-1}{2}}2^{j\frac{N+1}{2}} \|\tilde{\Delta}_jf\|_{L_\kappa^1(\mathbb{R}^n)},\\
\left\|W_t\tilde{\Delta}_jf\right\|_{L_\kappa^2(\mathbb{R}^n)} &\leq \|\tilde{\Delta}_jf\|_{L_\kappa^2(\mathbb{R}^n)}.
	\end{align*}
Applying the Riesz-Thorin interpolation theorem, for any $j\geq 0$, we have
\begin{align*}
	\|W_tS_0f\|_{L_\kappa^p (\mathbb{R}^n)}&\lesssim |t|^{-(N-1)\delta} \|S_0f\|_{L_\kappa^{p'}(\mathbb{R}^n)},\\
\|W_t\tilde{\Delta}_jf\|_{L_\kappa^p (\mathbb{R}^n)}&\lesssim |t|^{-(N-1)\delta}2^{j(N+1)\delta}\|\tilde{\Delta}_jf\|_{L_\kappa^{p'}(\mathbb{R}^n)}\leq |t|^{-(N-1)\delta}2^{-2js}\|\tilde{\Delta}_jf\|_{L_\kappa^{p'}(\mathbb{R}^n)},
	\end{align*}
 where in the last inequality we used the fact that $2s\leq -(N+1)\delta$ and $\delta=\frac{1}{2}-\frac{1}{p}$. Therefore
\begin{equation}\label{wave1}
    \|W_tf\|_{B_{p,q}^{s,\kappa}(\mathbb{R}^n)}\lesssim |t|^{-(N-1)\delta}\|f\|_{B_{p',q}^{-s,\kappa}(\mathbb{R}^n)},\,|t|\geq1.
\end{equation}
Next, we consider the case $|t|\leq 1$. Putting $\theta\geq 0$ in  \eqref{res-sum1} of Theorem \ref{ResultTime} and applying  Plancherel's identity and Riesz-Thorin interpolation theorem, we get
\begin{equation*}
   \|W_tS_0f\|_{L_\kappa^p (\mathbb{R}^n)}\lesssim \|S_0f\|_{L_\kappa^{p'}(\mathbb{R}^n)}.
\end{equation*}
Again for $j\geq0$,   from \eqref{res3-2} of Theorem \ref{ResultTime} with Plancherel's identity and Riesz-Thorin interpolation theorem, we obtain
\begin{equation*}
\|W_t\tilde{\Delta}_jf\|_{L_\kappa^p (\mathbb{R}^n)}\lesssim |t|^{-2\theta\delta}2^{2j(N-\theta)\delta}\|\tilde{\Delta}_jf\|_{L_\kappa^{p'}(\mathbb{R}^n)},
\end{equation*}   for $0\leq \theta\leq \frac{N-1}{2}$.
 Therefore 
\begin{equation*}
2^{js}\|W_t\tilde{\Delta}_jf\|_{L_\kappa^p (\mathbb{R}^n)}\lesssim |t|^{-2\theta\delta}2^{2j(s+(N-\theta)\delta)}2^{-js}\|\tilde{\Delta}_jf\|_{L_\kappa^{p'}(\mathbb{R}^n)}.
\end{equation*}
If $-N\delta \leq s\leq -\frac{(N+1)\delta}{2}$, then we can choose $0\leq \theta\leq \frac{N-1}{2}$ such that $s+N\delta=\theta \delta$. Furthermore, for $s<-N\delta$, we will choose $\theta=0$. Combing all the scenarios, we have
\begin{equation*}
2^{js}\|W_t\tilde{\Delta}_jf\|_{L_\kappa^p (\mathbb{R}^n)}\lesssim |t|^{-2\max (s+N\delta,0)}2^{-js}\|\tilde{\Delta}_jf\|_{L_\kappa^{p'}(\mathbb{R}^n)}
\end{equation*}
and hence
\begin{equation}\label{wave2}
\|W_tf\|_{B_{p,q}^{s,\kappa}(\mathbb{R}^n)}\lesssim |t|^{-2\max (s+N\delta,0)}\|f\|_{B_{p',q}^{-s,\kappa}(\mathbb{R}^n)}, 
\end{equation}
for $ |t|\leq 1$. 
Thus from \eqref{wave1} and \eqref{wave2}, we finally have
\begin{equation*}
\|W_tf\|_{B_{p,q}^{s,\kappa}(\mathbb{R}^n)}\lesssim k(t)\|f\|_{B_{p',q}^{-s,\kappa}(\mathbb{R}^n)}, \quad k(t)=\begin{cases}
	|t|^{-2\max (s+N\delta,0)}, \, &|t|\leq1,\\
	|t|^{-(N-1)\delta}, \,&|t|\geq1.
	\end{cases}
\end{equation*}
 \end{proof}
Now using   Proposition \ref{General-Strichartz}, we obtain the following estimates for the Dunkl-wave operator.
\begin{theorem}\label{Intermediate-wave} Let $N>1$. For $i=1,2$, let $p_i, r_i\in[2,\infty]$ and $s_i\in\mathbb{R}$ such that
	\begin{equation*}
	\begin{aligned}
	&(i)\,\frac{2}{r_i}+\frac{N-1}{p_i} \leq \frac{N-1}{2},   \\
	&(ii)\,s_i\leq -\frac{N+1}{2}\left(\frac{1}{2}-\frac{1}{p_i}\right),
	\end{aligned}
	\end{equation*}
 except for $(N, p_i, r_i)=(3,\infty,2)$. Then for any $T>0$, the Dunkl-wave operator $W_t$ satisfies 
	\begin{equation*}
	\|W_tu_0\|_{L^{r_1}(\mathbb{R}, B_{p_1,2}^{s_1,\kappa}(\mathbb{R}^n))}\lesssim\|u_0\|_{L_\kappa^2(\mathbb{R}^n)},
	\end{equation*}
	\begin{equation*}
	\left\|\int_0^t W_{t-\tau}g(\tau,\cdot)\,d\tau\right\|_{L^{r_1}\left([-T,T], B_{p_1,2}^{s_1,\kappa}(\mathbb{R}^n)\right)}\lesssim \|g\|_{L^{r'_2}\left([-T,T], B_{p'_2,2}^{-s_2,\kappa}(\mathbb{R}^n)\right)}.
	\end{equation*}
\end{theorem}
It immediately reduces the Strichartz inequality for the solution of the Dunkl-wave equation \eqref{Wave-equ}. 
\begin{corollary}\label{Besov-wave}
Under the same hypotheses as in Theorem \ref{Intermediate-wave}, the solution $u$ of  the Dunkl-wave equation \eqref{Wave-equ} satisfies the following estimate
	\begin{equation*}
	\|u\|_{L^{r_1}\left([-T,T], B_{p_1,2}^{s_1+1,\kappa}(\mathbb{R}^n)\right)} \lesssim \|u_0\|_{H^{1,\kappa}(\mathbb{R}^n)}+\|u_1\|_{L_\kappa^2(\mathbb{R}^n)}+\|g\|_{L^{r'_2}\left([-T,T], B_{p'_2,2}^{-s_2,\kappa}(\mathbb{R}^n)\right)}.
	\end{equation*}
\end{corollary}
Finally, by the embedding relation between the  Besov space and Lebesgue space in Proposition \ref{properties}, we have the following Strichartz estimates on Lebesgue spaces.
\begin{corollary}\label{WLebesgue} Let $N>1$. Suppose $r_i\geq 2$ and $2\leq p_i<\infty$ such that
\begin{equation*}
    \frac{1}{r_1}+\frac{N}{p_1}=\frac{N}{2}-1, \quad \frac{2}{r_1}+\frac{N-1}{p_1} \leq \frac{N-1}{2} \quad and \quad  \frac{1}{r_2}+\frac{N}{p_2}\leq \frac{N}{2},
\end{equation*}
except for $(N,p_1, r_1, p_2,r_2)=(3,\infty,2,3,2)$. If $u$ is the solution of the Dunkl-wave equation \eqref{Wave-equ}, for any $T>0$, we have the following estimate
	\begin{equation*}
	\|u\|_{L^{r_1}\left([-T,T],L_\kappa^{p_1}(\mathbb{R}^n)\right)} \lesssim \|u_0\|_{H^{1,\kappa}(\mathbb{R}^n)}+\|u_1\|_{L_\kappa^2(\mathbb{R}^n)}+\|g\|_{L^{r'_2}\left([-T,T],L_\kappa^{p_2'}(\mathbb{R}^n)\right)}.\\
	\end{equation*}
\end{corollary}
\begin{proof}
From Corollary \ref{Besov-wave}, it follows  that
\begin{equation*}
	\|u\|_{L^{r_1}\left([-T,T], B_{\tilde{p}_1,2}^{s_1+1,\kappa}(\mathbb{R}^n)\right)} \lesssim \|u_0\|_{H^{1,\kappa}(\mathbb{R}^n)}+\|u_1\|_{L_\kappa^2(\mathbb{R}^n)}+\|g\|_{L^{r'_2}\left([-T,T], B_{\tilde{p}'_2,2}^{-s_2,\kappa},(\mathbb{R}^n)\right)}
	\end{equation*}
 where $\tilde{p}_i, r_i\in[2,\infty]$ and
        \begin{equation*}
\frac{2}{r_1}+\frac{N-1}{\tilde{p}_1} = \frac{N-1}{2},\;\;\frac{2}{r_2}+\frac{N-1}{\tilde{p}_2} \leq \frac{N-1}{2}, ~~\text{ and }~~s_i=-\frac{N+1}{2}\left(\frac{1}{2}-\frac{1}{\tilde{p}_i}\right),\;i=1,2,
	\end{equation*}
except for $(N, \tilde{p}_i, {r}_i)=(3,\infty,2)$. If $p_i,r_i\in[2,\infty], i=1,2$ with the sole exception of the pair $(N,p_1, r_1, p_2,r_2)=(3,\infty,2,3,2)$ satisfies the condition 
\begin{equation*}
    \frac{1}{r_1}+\frac{N}{p_1}=\frac{N}{2}-1,\;\;\frac{2}{r_1}+\frac{N-1}{p_1} \leq \frac{N-1}{2}~~\text{ and }~~\frac{1}{r_2}+\frac{N}{p_2}\leq \frac{N}{2},
\end{equation*}
then we the following relations $$s_1+1\leq 0,\; \; s_1+1-\frac{N}{\tilde{p}_1}=0-\frac{N}{p_1},$$ and  $$-s_2\geq 0,\; \; -s_2-\frac{N}{\tilde{p}_2}=0-\frac{N}{p_2}.$$
Therefore, using the embedding relation between the Besov space and Lebesgue space in Proposition \ref{properties}, we obtain
$B_{\tilde{p}_1,2}^{s_1+1,\kappa}(\mathbb{R}^n)\subseteq L^{p_1}_\kappa(\mathbb{R}^n)$ and $L_\kappa^{p_2'}(\mathbb{R}^n)\subseteq B_{\tilde{p}'_2,2}^{-s_2,\kappa}(\mathbb{R}^n)$ and this  completes  our desired result.
\end{proof}

\subsection{Fractional Schr\"{o}dinger equation} Consider the fractional Dunkl-Schr\"{o}dinger equation ($0<\mu\leq 2$ and $\mu\neq 1$)
    \begin{equation}\label{FSchrEqu}
	\begin{cases}
	i\partial_tu+(-\Delta_\kappa)^\frac{\mu}{2} u=g,\\
	u|_{t=0}=u_0.
	\end{cases}
	\end{equation}
By Duhamel's principle, the solution is formally given by
	\begin{equation*}\label{solution1}
	u(t)=\mathcal{S}^\mu_tu_0-i\int_0^t\mathcal{S}^\mu_{t-\tau}g(\tau,\cdot)\,d\tau,
	\end{equation*}
where $\mathcal{S}^\mu_t=e^{it(-\Delta_\kappa)^\frac{\mu}{2}}$ and it corresponds to the case when $\phi(r)=r^\mu$. By a simple calculation, we see that 
	\begin{equation*}
	\phi'(r)=\mu r^{\mu-1}, \quad
	\phi''(r)=\mu(\mu-1)r^{\mu-2}.
	\end{equation*}
This shows that   $\phi$ satisfies (C1)-(C4) with $m_1=\alpha_1=m_2=\alpha_2=\mu$. Using  Theorem \ref{ResultTime}, we obtain the following decay estimate for $\mathcal{S}^\mu_t$.
\begin{proposition}
Let $0<\mu\leq 2$ and $\mu\neq 1$. Assume $2\leq p\leq\infty$, $1\leq q\leq\infty$, $\delta=\frac{1}{2}-\frac{1}{p}$ and $s\leq -N\left(1-\frac{\mu}{2}\right)\delta$, then we have
\begin{equation*}
    \|\mathcal{S}^\mu_tf\|_{B_{p,q}^{s,\kappa}(\mathbb{R}^n)}\lesssim k(t)\|f\|_{B_{p',q}^{-s,\kappa}(\mathbb{R}^n)}, \quad k(t)=\begin{cases}
	|t|^{-\frac{2}{\mu}\max (s+N\delta,0)}, \, &|t|\leq1,\\
	|t|^{-N\delta}, \,&|t|\geq1.
	\end{cases}
\end{equation*}
\end{proposition}
\begin{proof}
First, we prove the case $|t|\geq 1$. From \eqref{res-sum2} of Theorem \ref{ResultTime} with $\theta=\frac{N}{2}$,   Plancherel's identity yields
\begin{align*}
 \|\mathcal{S}^\mu_t S_0f\|_{L_\kappa^\infty (\mathbb{R}^n)}&\lesssim |t|^{-\frac{N}{2}}\|S_0f\|_{L_\kappa^1(\mathbb{R}^n)},\\
 \|\mathcal{S}^\mu_tS_0f\|_{L_\kappa^2 (\mathbb{R}^n)}&\leq \|S_0f\|_{L_\kappa^2(\mathbb{R}^n)}.
	\end{align*}
 Again from     \eqref{res3-3} of Theorem \ref{ResultTime} by taking $\theta=1$,   for $j\geq 0$,   Plancherel's identity gives
 \begin{align*}
\left\|\mathcal{S}^\mu_t\tilde{\Delta}_jf\right\|_{L_\kappa^\infty(\mathbb{R}^n)} &\lesssim |t|^{-\frac{N}{2}}2^{jN(1-\frac{\mu}{2})}\|\tilde{\Delta}_jf\|_{L_\kappa^1(\mathbb{R}^n)},\\
\left\|\mathcal{S}^\mu_t\tilde{\Delta}_jf\right\|_{L_\kappa^2(\mathbb{R}^n)} &\leq \|\tilde{\Delta}_jf\|_{L_\kappa^2(\mathbb{R}^n)}.
	\end{align*}
Now applying the Riesz-Thorin interpolation theorem, for any $j\geq 0$, we obtain
\begin{align*}
	\|\mathcal{S}^\mu_tS_0f\|_{L_\kappa^p (\mathbb{R}^n)}&\lesssim |t|^{-N\delta}\|S_0f\|_{L_\kappa^{p'}(\mathbb{R}^n)},\\
\|\mathcal{S}^\mu_t\tilde{\Delta}_jf\|_{L_\kappa^p (\mathbb{R}^n)}&\lesssim |t|^{-N\delta}2^{2jN\left(1-\frac{\mu}{2}\right)\delta}\|\tilde{\Delta}_jf\|_{L_\kappa^{p'}(\mathbb{R}^n)}\leq |t|^{-N\delta}2^{-2js}\|\tilde{\Delta}_jf\|_{L_\kappa^{p'}(\mathbb{R}^n)}.
	\end{align*}
This implies that  
\begin{equation*}
\|\mathcal{S}^\mu_tf\|_{B_{p,q}^{s,\kappa}(\mathbb{R}^n)}\lesssim |t|^{-N\delta}\|f\|_{B_{p',q}^{-s,\kappa}(\mathbb{R}^n)},\,|t|\geq1.
\end{equation*}
Next, we discuss the case $|t|\leq 1$. For $\theta=0$ in   \eqref{res-sum2} of Theorem \ref{ResultTime},  applying  $\theta=0$, Plancherel's identity and Riesz-Thorin interpolation theorem, we get
\begin{equation*}
   \|\mathcal{S}^\mu_tS_0f\|_{L_\kappa^p (\mathbb{R}^n)}\lesssim \|S_0f\|_{L_\kappa^{p'}(\mathbb{R}^n)}.
\end{equation*}
For $j\geq0$, it follows from \eqref{res3-2} 
 and \eqref{res3-3} of Theorem \ref{ResultTime},  Plancherel's identity and Riesz-Thorin interpolation theorem that for $0\leq \theta\leq \frac{N}{2}$, we have 
\begin{equation*}
\|\mathcal{S}^\mu_t\tilde{\Delta}_jf\|_{L_\kappa^p (\mathbb{R}^n)}\lesssim |t|^{-2\theta\delta}2^{2j(N-\theta\mu)\delta}\|\tilde{\Delta}_jf\|_{L_\kappa^{p'}(\mathbb{R}^n)}.
\end{equation*}
This implies that \begin{equation*}
2^{js}\|\mathcal{S}^\mu_t\tilde{\Delta}_jf\|_{L_\kappa^p (\mathbb{R}^n)}\lesssim |t|^{-2\theta\delta}2^{2j(s+(N-\theta\mu)\delta)}2^{-js}\|\tilde{\Delta}_jf\|_{L_\kappa^{p'}(\mathbb{R}^n)}.
\end{equation*}
If $-N\delta \leq s\leq -N\left(1-\frac{\mu}{2}\right)\delta$, then we can choose $0\leq \theta\leq \frac{N}{2}$ such that $s+N\delta=\mu\theta \delta$. If $s<-N\delta$, we take $\theta=0$. Therefore 
\begin{equation*}
2^{js}\|\mathcal{S}^\mu_t\tilde{\Delta}_jf\|_{L_\kappa^p (\mathbb{R}^n)}\lesssim |t|^{-\frac{2}{\mu}\max (s+N\delta,0)}2^{-js}\|\tilde{\Delta}_jf\|_{L_\kappa^{p'}(\mathbb{R}^n)}
\end{equation*}
and 
\begin{equation*}
\|\mathcal{S}^\mu_tf\|_{B_{p,q}^{s,\kappa}(\mathbb{R}^n)}\lesssim |t|^{-\frac{2}{\mu}\max (s+N\delta,0)}\|f\|_{B_{p',q}^{-s,\kappa}(\mathbb{R}^n)},\,|t|\leq 1,
\end{equation*}
which completes the proof of the proposition.
\end{proof}
By Proposition \ref{General-Strichartz}, we obtain the following estimates for the fractional Schr\"{o}dinger operator.
\begin{theorem}\label{Intermediate} Let $0<\mu\leq 2$ and $\mu\neq 1$. For $i=1,2$, let $p_i, r_i\in[2,\infty]$ and $s_i\in\mathbb{R}$ such that
	\begin{equation*}
	\begin{aligned}
	&(i)\,\frac{2}{r_i}+\frac{N}{p_i} \leq \frac{N}{2},   \\
	&(ii)\,s_i\leq -N\left(1-\frac{\mu}{2}\right)\left(\frac{1}{2}-\frac{1}{p_i}\right),
	\end{aligned}
	\end{equation*}
 except for $(N, p_i, r_i)=(2,\infty,2)$. Then for any $T>0$, the fractional Schr\"{o}dinger operator $\mathcal{S}^\mu_t$ satisfies the estimates
	\begin{equation}\label{Strichartz1}
	\|\mathcal{S}^\mu_tu_0\|_{L^{r_1}(\mathbb{R}, B_{p_1,2}^{s_1,\kappa}(\mathbb{R}^n))}\lesssim\|u_0\|_{L_\kappa^2(\mathbb{R}^n)},
	\end{equation}
	\begin{equation}\label{Strichartz2}
	\left\|\int_0^t\mathcal{S}^\mu_{t-\tau}g(\tau,\cdot)\,d\tau\right\|_{L^{r_1}\left([-T,T], B_{p_1,2}^{s_1,\kappa}(\mathbb{R}^n)\right)}\lesssim \|g\|_{L^{r'_2}\left([-T,T], B_{p'_2,2}^{-s_2,\kappa}(\mathbb{R}^n)\right)}.
	\end{equation}
\end{theorem}
The above estimate immediately reduces the following Strichartz inequality for the solution of the fractional Schr\"{o}dinger equation. 
\begin{corollary}
Under the same hypotheses as in Theorem \ref{Intermediate}, the solution $u$ of
the fractional Schr\"{o}dinger equation \eqref{FSchrEqu} satisfies the following estimate
	\begin{equation*}
	\|u\|_{L^{r_1}\left([-T,T], B_{p_1,2}^{s_1,\kappa}(\mathbb{R}^n)\right)} \lesssim \|u_0\|_{L_\kappa^2(\mathbb{R}^n)}+\|g\|_{L^{r'_2}\left([-T,T], B_{p'_2,2}^{-s_2,\kappa}(\mathbb{R}^n)\right)}.
	\end{equation*}
\end{corollary}
Finally, by the embedding relation between the  Besov space and Lebesgue space in Proposition \ref{properties}, we have the Strichartz inequalities on Lebesgue spaces.
\begin{corollary}\label{SLebesgue} Let $0<\mu<2$ and $\mu\neq 1$. Suppose $p_i,r_i\geq 2$ such that
\begin{equation*}
    \frac{1}{p_i}\geq \frac{1}{2}-\frac{2}{(2-\mu)N}\quad and \quad \frac{\mu}{r_i}+\frac{N}{p_i} \leq \frac{N}{2}-1.
\end{equation*}
If $u$ is the solution of the fractional Schr\"{o}dinger equation \eqref{FSchrEqu}, for any $T>0$, we have the following estimate
	\begin{equation*}
	\|u\|_{L^{r_1}\left([-T,T],L_\kappa^{p_1}(\mathbb{R}^n)\right)} \lesssim \|u_0\|_{H^{1,\kappa}(\mathbb{R}^n)}+\|g\|_{L^{r'_2}\left([-T,T],L_\kappa^{p_2'}(\mathbb{R}^n)\right)}.
	\end{equation*}
\end{corollary}
In particular, when $\mu=2$, fractional Schr\"{o}dinger operator reduces to the Dunkl-Schr\"odinger equation
\begin{equation}\label{D-SchrEqu}
	\begin{cases}
	i\partial_tu+\Delta_\kappa u=g,\\
	u|_{t=0}=u_0.
	\end{cases}
	\end{equation}
By the embedding relation between the  Besov space and Lebesgue space in Proposition \ref{properties} along with  Theorem \ref{Intermediate}, we obtain the Strichartz inequalities on Lebesgue space for the Dunkl-Schr\"odinger equation considered in \cite{Mejjaoli2009,Mejjaoli2013,shyam}.

\begin{corollary}\label{S-SLebesgue} Suppose for $j=1, 2$,  $r_i,p_i\geq 2$, $(N, p_i, r_i)\neq(2,\infty,2)$ and
\begin{equation*}
    \frac{2}{r_i}+\frac{N}{p_i} \leq \frac{N}{2}.
\end{equation*}
If $u$ is the solution of the Dunkl-Schr\"{o}dinger equation \eqref{D-SchrEqu}, then for any $T>0$, we have the following estimate
	\begin{equation*}
	\|u\|_{L^{r_1}\left([-T,T],L_\kappa^{p_1}(\mathbb{R}^n)\right)} \lesssim \|u_0\|_{L_\kappa^2(\mathbb{R}^n)}+\|g\|_{L^{r'_2}\left([-T,T],L_\kappa^{p_2'}(\mathbb{R}^n)\right)}.
	\end{equation*}
\end{corollary}
\subsection{Fourth-order Schr\"{o}dinger equation}
Consider the fourth-order Schr\"{o}dinger equation
	\begin{equation}\label{SSchrEqu}
	\begin{cases}
	i\partial_tu+\Delta_\kappa^2u-\Delta_\kappa u=g,\\
	u|_{t=0}=u_0.
	\end{cases}
	\end{equation}
By Duhamel's principle, the solution is formally given by
	\begin{equation*}\label{solution3}
	u(t)=\mathcal{U}_tu_0-i\int_0^t\mathcal{U}_{t-\tau}g(\tau,\cdot)\,d\tau,
	\end{equation*}
where $\mathcal{U}_t=e^{it(\Delta_\kappa^2-\Delta_\kappa)}$ and it corresponds to the case when $\phi(r)=r^4+r^2$. By a simple calculation, we see that 
	\begin{equation*}
	\phi'(r)=4r^3+2r, \quad
	\phi''(r)=12r^2+2.
	\end{equation*}
Thus  $\phi$ satisfies (C1)–(C4) with $m_1=\alpha_1=4$, $m_2=\alpha_2=2$.  Using  Theorem \ref{ResultTime}, we obtain the following decay estimate for  $\mathcal{U}_t$. 
\begin{proposition}
Assume $2\leq p\leq\infty$, $1\leq q\leq\infty$, $\delta=\frac{1}{2}-\frac{1}{p}$ and $s\leq N\delta$.\\
Then we have
\begin{equation*}
    \|\mathcal{U}_tf\|_{B_{p,q}^{s,\kappa}(\mathbb{R}^n)}\lesssim k(t)\|f\|_{B_{p',q}^{-s,\kappa}(\mathbb{R}^n)}, \quad k(t)=\begin{cases}
	|t|^{-\frac{1}{2}\max (s+N\delta,0)} , \, &|t|\leq1,\\
	|t|^{-N\delta}, \,&|t|\geq1.
	\end{cases}
\end{equation*}
\end{proposition}
\begin{proof}
First, we prove the case $|t|\geq 1$. For $\theta=\frac{N}{2}$ in  \eqref{res-sum2} of Theorem \ref{ResultTime},  Plancherel's identity gives 
\begin{align*}
 \|\mathcal{U}_tS_0f\|_{L_\kappa^\infty (\mathbb{R}^n)}&\lesssim |t|^{-\frac{N}{2}}\|S_0f\|_{L_\kappa^1 (\mathbb{R}^n)},\\
\|\mathcal{U}_tS_0f\|_{L_\kappa^2 (\mathbb{R}^n)}&\leq\|S_0f\|_{L_\kappa^2 (\mathbb{R}^n)}.
	\end{align*}
Applying  Riesz-Thorin interpolation theorem, we have
\begin{equation}\label{U-Sf}
  \|\mathcal{U}_tS_0f\|_{L_\kappa^p (\mathbb{R}^n)}\lesssim |t|^{-N\delta}\|S_0f\|_{L_\kappa^{p'} (\mathbb{R}^n)}.
\end{equation}
Again for $j\geq 0$, it follows from \eqref{res3-2} and \eqref{res3-3} of Theorem \ref{ResultTime} that
\begin{equation}\label{full-U}
	\|\mathcal{U}_t \tilde{\Delta}_j f|_{L_\kappa^\infty(\mathbb{R}^n)} \lesssim |t|^{-\theta}2^{j\left(N-4\theta\right)}\|\tilde{\Delta}_j f\|_{L_\kappa^1(\mathbb{R}^n)},\,0 \leq \theta \leq \frac{N}{2}.
	\end{equation}
Taking $\theta=\frac{N}{2}$ in \eqref{full-U}, combined with Plancherel's identity and Riesz-Thorin interpolation theorem, it indicates
\begin{equation}\label{j-Uf}
   \|\mathcal{U}_t\tilde{\Delta}_jf\|_{L_\kappa^p (\mathbb{R}^n)}\lesssim |t|^{-N\delta}2^{-2jN\delta}\|\tilde{\Delta}_jf\|_{L_\kappa^{p'} (\mathbb{R}^n)}\leq |t|^{-N\delta}2^{-2js}\|\tilde{\Delta}_jf\|_{L_\kappa^{p'}},
\end{equation}
for any $j\geq0$. From \eqref{U-Sf} and \eqref{j-Uf}, we obtain our desired result of the proposition in this case.
 
Next, we discuss the case $|t|\leq 1$. For $\theta=0$ in  \eqref{res-sum1} of Theorem \ref{ResultTime},  Plancherel's identity and Riesz-Thorin interpolation theorem gives 
\begin{equation*}
    \|\mathcal{U}_tS_0f\|_{L^p_\kappa(\mathbb{R}^n)}\lesssim \|S_0f\|_{L^{p'}_\kappa(\mathbb{R}^n)}.
\end{equation*}
Again for $j\geq0$,   from \eqref{full-U}, applying Plancherel's identity and Riesz-Thorin interpolation theorem,  for $0\leq \theta\leq \frac{N}{2}$, we obtain
\begin{equation*}
\|\mathcal{U}_t\tilde{\Delta}_jf\|_{L^p_\kappa(\mathbb{R}^n)}\lesssim |t|^{-2\theta\delta}2^{2j(N-4\theta)\delta}\|\tilde{\Delta}_jf\|_{L^{p'}_\kappa(\mathbb{R}^n)},
\end{equation*}
which reduces to
\begin{equation*}
2^{js}\|\mathcal{U}_t\tilde{\Delta}_jf\|_{L^p_\kappa(\mathbb{R}^n)}\lesssim |t|^{-2\theta\delta}2^{2j\left(s+(N-4\theta)\delta\right)}2^{-js}\|\tilde{\Delta}_jf\|_{L^{p'}_\kappa(\mathbb{R}^n)}.
\end{equation*}
If $-N\delta \leq s \leq N\delta$, then we can choose $0\leq \theta\leq \frac{N}{2}$ such that $s+N\delta=4\theta \delta$. If $s<-N\delta$, we take $\theta=0$. Therefore 
\begin{equation*}
2^{js}\|\mathcal{U}_t\tilde{\Delta}_jf\|_{L^p_\kappa(\mathbb{R}^n)}\lesssim |t|^{-\frac{1}{2}\max (s+N\delta,0)}2^{-js}\|\tilde{\Delta}_jf\|_{L^{p'}_\kappa(\mathbb{R}^n)},
\end{equation*}
and hence
\begin{equation*}
\|\mathcal{U}_tf\|_{B_{p,q}^{s,\kappa}(\mathbb{R}^n)}\lesssim |t|^{-\frac{1}{2}\max (s+N\delta,0)}\|f\|_{B_{p,q}^{-s,\kappa}(\mathbb{R}^n)},\,|t|\leq 1.
\end{equation*}
This completes the proof of the proposition.
\end{proof}
\subsection{Klein-Gordon equation}
Moreover, we consider the Klein-Gordon equation
	\begin{equation}\label{K-GEqu}
	\begin{cases}
	\partial_t^2u-\Delta_\kappa u+u=g,\\
	u|_{t=0}=u_0 ,  \\
	\partial_tu|_{t=0}=u_1.
	\end{cases}
	\end{equation}
By Duhamel's principle, the solution is formally given by
\begin{equation*}\label{solution3}
	u(t)=\frac{d\mathcal{K}_t}{dt}u_0+\mathcal{K}_tu_1-\int_0^t\mathcal{K}_{t-\tau}g(\tau,\cdot)\,d\tau,
	\end{equation*}
where
	\begin{equation*}
	\mathcal{K}_t=\frac{\sin(t\sqrt{I-\Delta_\kappa})}{\sqrt{I-\Delta_\kappa}},
	\quad  \frac{d\mathcal{K}_t}{dt}=\cos(t\sqrt{I-\Delta_\kappa}).
	\end{equation*}
So we naturally introduce the operator $K_t=e^{it\sqrt{I-\Delta_\kappa}}$, which corresponds to the case when $\phi(r)=\sqrt{1+r^2}$.
By a simple calculation have
	\begin{equation*}
	\phi'(r)=r(1+r^2)^{-\frac{1}{2}}, \;
	\phi''(r)=(1+r^2)^{-\frac{3}{2}}.
	\end{equation*}
This shows that    $\phi$ satisfies (C1)–(C4) with $m_1=1,\alpha_1=-1, $ and $m_2=\alpha_2=2$. Using  Theorem \ref{ResultTime}, we obtain the following decay estimate for $K_t$.
\begin{proposition}\label{Klein-decay}
Assume $2\leq p\leq\infty$, $1\leq q\leq\infty$, and $\delta=\frac{1}{2}-\frac{1}{p}$.
\begin{enumerate}
    \item Let $0\leq\theta\leq 1$ and $2s=-\left(N+1+\theta\right)\delta$, then we have
\begin{equation*}
    \|K_tf\|_{B_{p,q}^{s,\kappa}(\mathbb{R}^n)}\lesssim |t|^{-(N-1+\theta)\delta}\|f\|_{B_{p',q}^{-s,\kappa}(\mathbb{R}^n)}.
\end{equation*}
\item Let $0\leq\theta\leq N-1$ and $2s=-\left(N+1+\theta\right)\delta$, then we have
\begin{equation*}
    \|K_tf\|_{B_{p,q}^{s,\kappa}(\mathbb{R}^n)}\lesssim |t|^{-(N-1-\theta)\delta}\|f\|_{B_{p',q}^{-s,\kappa}(\mathbb{R}^n)}.
\end{equation*}
\item In particular, for $0\leq\theta\leq 1$ and $2s\leq  -\left(N+1+\theta\right)\delta$,   we have
\begin{equation*}
    \|K_tf\|_{B_{p,q}^{s,\kappa}(\mathbb{R}^n)}\lesssim k(t)\|f\|_{B_{p',q}^{-s,\kappa}(\mathbb{R}^n)}, \quad k(t)=\begin{cases}
	|t|^{-2\max (s+N\delta,0)} , \, &|t|\leq1,\\
	|t|^{-(N-1+\theta)\delta}, \,&|t|\geq1.
	\end{cases}
\end{equation*}
\end{enumerate}
 
\end{proposition}
\begin{proof}
First, we prove part (1). From \eqref{res3-3} and \eqref{res-sum2} of Theorem \ref{ResultTime} with  Plancherel's identity and Riesz-Thorin interpolation theorem, for $0\leq \theta\leq 1$, we obtain
\begin{align*}
   \|
   K_t\tilde{\Delta}_jf\|_{L_\kappa^p (\mathbb{R}^n)}&\lesssim |t|^{-(N-1+\theta)\delta}2^{j\left(N+1+\theta\right)\delta}\|\tilde{\Delta}_jf\|_{L_\kappa^{p'}(\mathbb{R}^n)}, \quad  j\geq 0,\\
    \|K_tS_0f\|_{L_\kappa^p (\mathbb{R}^n)}&\lesssim |t|^{-(N-1+\theta)\delta}\|S_0f\|_{L_\kappa^{p'}(\mathbb{R}^n)}.
\end{align*}
Hence,  we complete the proof of part (1)  from the fact that   $2s=-\left(N+1+\theta\right)\delta$.
 
 Next, we prove part (2). Again from \eqref{res3-2} and \eqref{res-sum1} of Theorem \ref{ResultTime} with the help of  Plancherel's identity and Riesz-Thorin interpolation theorem, for $0\leq \theta\leq N-1$, we get
\begin{align*}
     \|K_tS_0f\|_{L_\kappa^p (\mathbb{R}^n)}&\lesssim |t|^{-(N-1-\theta)\delta}\|S_0f\|_{L_\kappa^{p'}(\mathbb{R}^n)},\\
    \|K_t\tilde{\Delta}_jf\|_{L_\kappa^p (\mathbb{R}^n)}&\lesssim |t|^{-(N-1-\theta)\delta}2^{j(N+1+\theta)\delta}\|\tilde{\Delta}_jf\|_{L_\kappa^{p'}(\mathbb{R}^n)}.
\end{align*}
From the fact that $2s=-(N+1+\theta)\delta$, we have
\begin{equation*}
    \|K_tf\|_{B_{p,q}^{s,\kappa}(\mathbb{R}^n)}\lesssim |t|^{-(N-1-\theta)\delta}\|f\|_{B_{p',q}^{-s,\kappa}(\mathbb{R}^n)}
\end{equation*}
and this completes the proof of part (2). To prove (3), we divide into the following three cases. 

\textbf{Case 1.} $|t|\geq 1$. Following the proof of part (1), we immediately get the required estimates. 

\textbf{Case 2.} $|t|\leq 1$ and $s+N\delta\geq 0$. In this case,  our required estimate will follow immediately from part (2). Indeed, if $2s\leq -(N+1+\theta)\delta$ for some $0\leq\theta\leq1$, then $0\leq 2(s+N\delta)\leq (N-1-\theta)\delta$, which also provides  us $0\leq \theta \leq N-1$. Therefore from part (2), we obtain 
\begin{equation*}
\|K_tf\|_{B_{p,q}^{s,\kappa}(\mathbb{R}^n)}\lesssim |t|^{-(N-1-\theta)\delta}\|f\|_{B_{p',q}^{-s,\kappa}(\mathbb{R}^n)}\leq |t|^{-2(s+N\delta)}\|f\|_{B_{p',q}^{-s,\kappa}(\mathbb{R}^n)}.
\end{equation*}  
\textbf{Case 3.} $|t|\leq 1$ and $s+N\delta< 0$. In this case taking $\theta_1=N-1$ in (2), for $s_1=-N\delta>s$, we get
\begin{equation*}
    \|K_tf\|_{B_{p,q}^{s_1,\kappa}(\mathbb{R}^n)}\lesssim \|f\|_{B_{p',q}^{-s_1,\kappa}(\mathbb{R}^n)}.
\end{equation*}
Since $s_1>s$,  using the inclusion relation (4) between Besov spaces in Proposition \ref{properties}, we   immediately obtain 
\begin{equation*}
   \|K_tf\|_{B_{p,q}^{s,\kappa}(\mathbb{R}^n)}\|K_tf\|_{B_{p,q}^{s_1,\kappa}(\mathbb{R}^n)}\lesssim \|f\|_{B_{p',q}^{-s_1,\kappa}(\mathbb{R}^n)}\leq \|f\|_{B_{p',q}^{-s,\kappa}(\mathbb{R}^n)}
\end{equation*} and this completes the proof of part (3). 
\end{proof}

 \subsection{Beam equation}
Finally, we consider the beam equation
	\begin{equation}\label{BeamEqu}
	\begin{cases}
	\partial_t^2u+\Delta_\kappa^2 u+u=g,\\
	u|_{t=0}=u_0 ,  \\
	\partial_tu|_{t=0}=u_1.
	\end{cases}
	\end{equation}
By Duhamel's principle, the solution is formally given by
\begin{equation*}\label{solution3}
	u(t)=\frac{d\mathcal{B}_t}{dt}u_0+\mathcal{B}_tu_1-\int_0^t\mathcal{B}_{t-\tau}g(\tau,\cdot)\,d\tau,
	\end{equation*}
where
	\begin{equation*}
	\mathcal{B}_t=\frac{\sin(t\sqrt{I+\Delta_\kappa^2})}{\sqrt{I+\Delta_\kappa^2}},
	\quad  \frac{d\mathcal{B}_t}{dt}=\cos(t\sqrt{I+\Delta_\kappa^2}).
	\end{equation*}
So we naturally introduce the operator $B_t=e^{it\sqrt{I+\Delta_k^2}}$, which corresponds to the case when $\phi(r)=\sqrt{1+r^4}$.
By a simple calculation, we see that
	\begin{equation*}
	\phi'(r)=2r^3(1+r^4)^{-\frac{1}{2}}, \quad
	\phi''(r)=(6r^2+2r^6)(1+r^4)^{-\frac{3}{2}}.
	\end{equation*}
This shows that  $\phi$ satisfies (C1)–(C4) with $m_1=\alpha_1=2, m_2=\alpha_2=4$.  Using  Theorem \ref{ResultTime}, we obtain the following decay estimate for $B_t$.
\begin{proposition}\label{beam-decay}
Let $2\leq p\leq\infty$, $1\leq q\leq\infty$, $\delta=\frac{1}{2}-\frac{1}{p},$ and $s\leq 0$. Then 
\begin{equation*}
    \|B_tf\|_{B_{p,q}^{s,\kappa}(\mathbb{R}^n)}\lesssim k(t)\|f\|_{B_{p',q}^{-s,\kappa}(\mathbb{R}^n)}, \quad k(t)=\begin{cases}
	|t|^{-\max (s+N\delta,0)} , \, &|t|\leq1,\\
	|t|^{-\frac{N\delta}{2}}, \,&|t|\geq1.
	\end{cases}
\end{equation*}
\end{proposition}
\begin{proof}
First, we prove the case $|t|\geq 1$. For $\theta=\frac{N}{4}$ in  \eqref{res-sum2} of Theorem \ref{ResultTime} with the help of Plancherel's identity and Riesz-Thorin interpolation theorem, we have
\begin{equation}\label{B-Sf}
    \|B_t S_0f\|_{L_\kappa^p (\mathbb{R}^n)}\lesssim |t|^{-\frac{N\delta}{2}}\|S_0f\|_{L_\kappa^{p'} (\mathbb{R}^n)}.
\end{equation}
Again for $j\geq 0$, it follows from \eqref{res3-2} and  \eqref{res3-3} of Theorem \ref{ResultTime} that
\begin{equation}\label{full-B}
	\|B_t \tilde{\Delta}_j f\|_{L_\kappa^\infty(\mathbb{R}^n)} \lesssim |t|^{-\theta}2^{j\left(N-2\theta\right)}\|\tilde{\Delta}_jf\|_{L_\kappa^1(\mathbb{R}^n)},\,0 \leq \theta \leq \frac{N}{2}.
	\end{equation}
Taking $\theta=\frac{N}{2}$ in \eqref{full-B}, combined with Plancherel's identity and Riesz-Thorin interpolation theorem, we obtain
\begin{equation}\label{j-Bf}
\|B_t\tilde{\Delta}_j f\|_{L_\kappa^p (\mathbb{R}^n)}\lesssim |t|^{-N\delta}\|\tilde{\Delta}_j f\|_{L_\kappa^{p'} (\mathbb{R}^n)}\leq|t|^{-\frac{N\delta}{2}}2^{-2js}\|\tilde{\Delta}_jf\|_{L_\kappa^{p'} (\mathbb{R}^n)},
\end{equation}
for any $j\geq0$. From \eqref{B-Sf} and \eqref{j-Bf}, we get our desired result.

Next, we consider the case $|t|<1$. By taking $\theta=0$ in   \eqref{res-sum2} of Theorem \ref{ResultTime} along with   Plancherel's identity and Riesz-Thorin interpolation theorem, we have
\begin{equation*}
    \|B_t S_0f\|_{L_\kappa^p (\mathbb{R}^n)}\lesssim \|S_0f\|_{L_\kappa^{p'} (\mathbb{R}^n)}.
\end{equation*}
For $j\geq 0$, it follows from \eqref{full-B}, Plancherel's identity, and Riesz-Thorin interpolation theorem, for $0 \leq \theta \leq \frac{N}{2}$, we have   
\begin{equation*}
	\|B_t \tilde{\Delta}_j f\|_{L_\kappa^p(\mathbb{R}^n)} \lesssim |t|^{-2\theta\delta}2^{2j\left(N-2\theta\right)\delta}\|\tilde{\Delta}_j f\|_{L_\kappa^{p'}(\mathbb{R}^n)},
	\end{equation*}
and thus
\begin{equation*}
	2^{js}\|B_t \tilde{\Delta}_j f\|_{L_\kappa^p(\mathbb{R}^n)} \lesssim |t|^{-2\theta\delta}2^{2j\left(s+\left(N-2\theta\right)\delta\right)}2^{-js}\|\tilde{\Delta}_j f\|_{L_\kappa^{p'}(\mathbb{R}^n)}.
	\end{equation*}
If $-N\delta \leq s\leq 0$, we can choose $0\leq\theta\leq \frac{N}{2}$ such that $s+N\delta=2\theta\delta$. If $s<-N\delta$, we can choose $\theta=0$. Thus, we have
\begin{equation*}
	2^{js}\|B_t \tilde{\Delta}_j f\|_{L_\kappa^p(\mathbb{R}^n)} \lesssim |t|^{-\max (s+N\delta,0)}2^{-js}\|\tilde{\Delta}_j f\|_{L_\kappa^{p'}(\mathbb{R}^n)},
	\end{equation*}
which completes the proof of the proposition.
\end{proof}

\section{Well-posedness results for nonlinear Klein-Gordon and beam equations}\label{sec6}
Using the decay estimates for the nonlinear Klein-Gordon and beam equation with the different decay rates between $|t|>1$ and $|t|\leq1$ in the previous section, we prove the global well-posedness results for nonlinear problems. First, we consider the Cauchy problem for the nonlinear Klein-Gordon equation
\begin{equation}\label{N-K-GEqu}
	\begin{cases}
	\partial_t^2u-\Delta_\kappa u+u=F(u),\\
	u|_{t=0}=u_0 ,  \\
	\partial_tu|_{t=0}=u_1,
	\end{cases}
	\end{equation}
with a power-like nonlinearity $F(u)$ satisfies  
\begin{equation}\label{F-condition}
    F(0)=0,\quad |F(u)-F(v)|\leq C(|u|^\alpha+|v|^\alpha)|u-v|
\end{equation}
for some $\alpha>0$, where $C$ is a positive constant.
By Duhamel's principle, the solution of \eqref{N-K-GEqu}   is formally given by
\begin{equation*}
    u=\frac{d\mathcal{K}_t}{dt}u_0+\mathcal{K}_tu_1-\int_0^t\mathcal{K}_{t-\tau}F(u)\,d\tau,
\end{equation*}
where $\mathcal{K}_t=\frac{\sin(t\sqrt{I-\Delta_\kappa})}{\sqrt{I-\Delta_\kappa}}$.  Similarly, as in the Laplacian in   Euclidean space, we find the critical index as
\begin{equation*}
    \alpha_K(N)=\frac{2-N+\sqrt{N^2+12N+4}}{2N}.
\end{equation*}
Using the classical fixed point argument, we obtain the global well-posedness theorem for the nonlinear Klein-Gordon equation.
\begin{theorem}
     Let $\alpha_K(N)<\alpha<\frac{4}{N}$, $\sigma(\alpha)=\frac{\alpha(N+2)}{2(\alpha+2)}$, and $(u_0,u_1)\in H^{\sigma(\alpha),\kappa}_\frac{\alpha+2}{\alpha+1}(\mathbb{R}^n)\times H^{\sigma(\alpha)-1,\kappa}_\frac{\alpha+2}{\alpha+1}(\mathbb{R}^n)$ with sufficient small norm. Then the equation \eqref{N-K-GEqu} has a unique solution
     \begin{equation*}
         u\in C(\mathbb{R}, H^{\frac{\alpha}{\alpha+2}, \kappa}(\mathbb{R}^n))\cap L^{\alpha+1}(\mathbb{R},L_\kappa^{\alpha+2}(\mathbb{R}^n)).
     \end{equation*}
\end{theorem}
\begin{proof}
 By Proposition \ref{properties} and \ref{Klein-decay}, we have
 \begin{align}
     \|\frac{d\mathcal{K}_t}{dt} u_0\|_{L^{\alpha+2}_k(\mathbb{R}^n)}&\leq  \|\frac{d\mathcal{K}_t}{dt} u_0\|_{B^{0,\kappa}_{\alpha+2,2}(\mathbb{R}^n)}\lesssim k(t) \|u_0\|_{B^{\sigma(\alpha),\kappa}_{\frac{\alpha+2}{\alpha+1},2}(\mathbb{R}^n)}\lesssim k(t) \|u_0\|_{H^{\sigma(\alpha),\kappa}_\frac{\alpha+2}{\alpha+1}(\mathbb{R}^n)},\nonumber\\ 
      \|\mathcal{K}_t u_1\|_{L^{\alpha+2}_k(\mathbb{R}^n)}&\leq  \|\mathcal{K}_t u_1\|_{B^{0,\kappa}_{\alpha+2,2}(\mathbb{R}^n)}\lesssim k(t) \|u_1\|_{B^{\sigma(\alpha)-1,\kappa}_{\frac{\alpha+2}{\alpha+1},2}(\mathbb{R}^n)}\lesssim k(t) \|u_1\|_{H^{\sigma(\alpha)-1,\kappa}_\frac{\alpha+2}{\alpha+1}(\mathbb{R}^n)},\label{A-decay}
 \end{align}
 where
 \begin{equation*}
 k(t)=\begin{cases}
	|t|^{-\max (\frac{\alpha(N-2)}{2(\alpha+2)},0)} , \, &|t|\leq1,\\
	|t|^{-\frac{\alpha N}{2(\alpha+2)}}, \,&|t|\geq1.
	\end{cases}
\end{equation*}
For   $\alpha_K(N)<\alpha<\frac{4}{N}$, we notice that 
\begin{equation*}
    (\alpha+1)\frac{\alpha(N-2)}{2(\alpha+2)}<1,\; \;(\alpha+1)\frac{\alpha N}{2(\alpha+2)}>1,\;\;\text{and }\; \sigma(\alpha)<1.
\end{equation*}
This implies that  $k\in L^{\alpha+1}(\mathbb{R})$ and  
\begin{align*}
  \|\frac{d\mathcal{K}_t}{dt} u_0\|_{L^{\alpha+1}(\mathbb{R},L^{\alpha+2}_k(\mathbb{R}^n))}&\leq C\|u_0\|_{H^{\sigma(\alpha),\kappa}_\frac{\alpha+2}{\alpha+1}(\mathbb{R}^n)},\\
 \|\mathcal{K}_t u_1\|_{L^{\alpha+1}(\mathbb{R},L^{\alpha+2}_k(\mathbb{R}^n))}&\leq C\|u_1\|_{H^{\sigma(\alpha)-1,\kappa}_\frac{\alpha+2}{\alpha+1}(\mathbb{R}^n)}.
\end{align*}
Therefore,  from \eqref{A-decay}, Proposition \ref{properties}, and Young's inequality, it follows  that
\begin{align*}
 \left\|\int_0^t\mathcal{K}_{t-\tau}F(u)\,d\tau\right\|_{L^{\alpha+1}(\mathbb{R},L^{\alpha+2}_k(\mathbb{R}^n))}&\leq  \left\|\int_0^t\|\mathcal{K}_{t-\tau}F(u)\|_{L^{\alpha+2}_k(\mathbb{R}^n)}\,d\tau\right\|_{L^{\alpha+1}(\mathbb{R})} \\
 &\leq C \left\|\int_0^t k(t-\tau)\|F(u)\|_{H^{\sigma(\alpha)-1,\kappa}_\frac{\alpha+2}{\alpha+1}(\mathbb{R}^n)}\,d\tau\right\|_{L^{\alpha+1}(\mathbb{R})} \\
  &\leq C \left\|\int_0^t k(t-\tau)\|F(u)\|_{L_k^\frac{\alpha+2}{\alpha+1}(\mathbb{R}^n)}\,d\tau\right\|_{L^{\alpha+1}(\mathbb{R})} \\
  &\leq C\|k\|_{L^{\alpha+1}(\mathbb{R})} \|F(u)\|_{L^1(\mathbb{R},L_k^\frac{\alpha+2}{\alpha+1}(\mathbb{R}^n))}\\
  &\leq   C\left\||u|^{\alpha+1}\right\|_{L^1(\mathbb{R},L_k^\frac{\alpha+2}{\alpha+1}(\mathbb{R}^n))}\\
  &=C\|u\|^{\alpha+1}_{L^{\alpha+1}(\mathbb{R},L_k^{\alpha+2}(\mathbb{R}^n))}.
\end{align*}
Taking $M=2C(\|u_0\|_{H^{\sigma(\alpha),\kappa}_\frac{\alpha+2}{\alpha+1}(\mathbb{R}^n)}+\|u_1\|_{H^{\sigma(\alpha)-1,\kappa}_\frac{\alpha+2}{\alpha+1}(\mathbb{R}^n)})$, we set the evolution space $X$ as 
\begin{equation*}
    X=\{u\in L^{\alpha+1}(\mathbb{R},L_k^{\alpha+2}(\mathbb{R}^n)): \|u\|_{L^{\alpha+1}(\mathbb{R},L_k^{\alpha+2}(\mathbb{R}^n))}\leq M\}.
\end{equation*}
We say that a function  $u$ is  a {\it mild solution} to (\ref{N-K-GEqu})  if $u$ is a fixed point for  the    integral operator  
\begin{equation*}
    \mathcal{T}: u\in X \mapsto \mathcal{T}u(t, x):=\frac{d\mathcal{K}_t}{dt}u_0+A_tu_1-\int_0^t\mathcal{K}_{t-\tau}F(u)\,d\tau. 
\end{equation*}
We prove the global-in-time existence and uniqueness of small data Sobolev solutions of low regularity to the Cauchy problem (\ref{N-K-GEqu})  by finding a unique fixed point to the operator $ \mathcal{T}$. It means that  there exists a unique global solution $u$    to the equation $ \mathcal{T} u=u\in X$, which also gives the solution to (\ref{N-K-GEqu}). 	In order to prove that the operator $ \mathcal{T}$ has a uniquely determined fixed point, we use Banach's fixed point argument with respect to the norm on $X$ as defined above.    

 Now for any $u,v\in X$, we have
\begin{equation*}
    \|\mathcal{T}u\|_{L^{\alpha+1}(\mathbb{R},L_k^{\alpha+2}(\mathbb{R}^n))}\leq \frac{M}{2}+CM^{\alpha+1},
\end{equation*}
and 
\begin{align*}
      &\quad\|\mathcal{T}u-\mathcal{T}v\|_{L^{\alpha+1}(\mathbb{R},L_k^{\alpha+2}(\mathbb{R}^n))}\\
     &=\left\|\int_0^t\mathcal{K}_{t-\tau}\left(F(u)-F(v)\right)\,d\tau\right\|_{L^{\alpha+1}(\mathbb{R},L_k^{\alpha+2}(\mathbb{R}^n))}\\
     &\leq C \left\|F(u)-F(v)\right\|_{L^1(\mathbb{R},L_k^\frac{\alpha+2}{\alpha+1}(\mathbb{R}^n))}\\
     &\leq C \left\|(|u|^\alpha+|v|^\alpha)|u-v|\right\|_{L^1(\mathbb{R},L_k^\frac{\alpha+2}{\alpha+1}(\mathbb{R}^n))}\\
    &=C\left\|\left\|(|u|^\alpha+|v|^\alpha)|u-v|\right\|_{L_k^\frac{\alpha+2}{\alpha+1}(\mathbb{R}^n)}\right\|_{L^1(\mathbb{R})}\\
     &\leq C\left\|\left\||u|^\alpha+|v|^\alpha\right\|_{L_k^\frac{\alpha+2}{\alpha}(\mathbb{R}^n)}\left\|u-v\right\|_{L_k^{\alpha+2}(\mathbb{R}^n)}\right\|_{L^1(\mathbb{R})}\\
      &\leq C\left\||u|^\alpha+|v|^\alpha\right\|_{L^\frac{\alpha+1}{\alpha}(\mathbb{R}, L_k^\frac{\alpha+2}{\alpha}(\mathbb{R}^n))}\left\|u-v\right\|_{L^{\alpha+1}(\mathbb{R}, L_k^{\alpha+2}(\mathbb{R}^n))}\\
       &\leq C\left(\left\||u|^\alpha\right\|_{L^\frac{\alpha+1}{\alpha}(\mathbb{R}, L_k^\frac{\alpha+2}{\alpha}(\mathbb{R}^n))}+\left\||v|^\alpha\right\|_{L^\frac{\alpha+1}{\alpha}(\mathbb{R}, L_k^\frac{\alpha+2}{\alpha}(\mathbb{R}^n))}\right)\left\|u-v\right\|_{L^{\alpha+1}(\mathbb{R}, L_k^{\alpha+2}(\mathbb{R}^n))}\\
       &=C\left(\|u\|_{L^{\alpha+1}(\mathbb{R}, L_k^{\alpha+2}(\mathbb{R}^n))}^\alpha+\|v\|_{L^{\alpha+1}(\mathbb{R}, L_k^{\alpha+2}(\mathbb{R}^n))}^\alpha \right)\left\|u-v\right\|_{L^{\alpha+1}(\mathbb{R}, L_k^{\alpha+2}(\mathbb{R}^n))}\\
       &\leq 2CM^\alpha \left\|u-v\right\|_{L^{\alpha+1}(\mathbb{R}, L_k^{\alpha+2}(\mathbb{R}^n))}.
\end{align*}
Notice that, for    $CM^\alpha<\frac{1}{2}$, $\mathcal{T}: X\to X$ is a constraction mapping. Applying Banach's fixed point theorem, there exists a global small data Sobolev solution $u$ of the equation $ u=\mathcal{T}u$ in $ X$,  which also gives the solution to the equation \eqref{N-K-GEqu}. In addition, by the embedding $ H^{\sigma(\alpha),\kappa}_\frac{\alpha+2}{\alpha+1}(\mathbb{R}^n)\subseteq H^{\frac{\alpha }{\alpha+2},\kappa}(\mathbb{R}^n)$ and $ H^{\sigma(\alpha)-1,\kappa}_\frac{\alpha+2}{\alpha+1}(\mathbb{R}^n)\subseteq H^{-\frac{2}{\alpha+2},\kappa}(\mathbb{R}^n)$, we have
\begin{align*}
    \|u\|_{L^\infty(\mathbb{R}, H^{\frac{\alpha}{\alpha+2}, \kappa}(\mathbb{R}^n))}
    &\leq  \|\frac{d\mathcal{K}_t}{dt}u_0\|_{L^\infty(\mathbb{R}, H^{\frac{\alpha}{\alpha+2}, \kappa}(\mathbb{R}^n))}+ \|\mathcal{K}_tu_1\|_{L^\infty(\mathbb{R}, H^{\frac{\alpha}{\alpha+2}, \kappa}(\mathbb{R}^n))} \\
    &\qquad \qquad+\|\int_0^t\mathcal{K}_{t-\tau}F(u)\,d\tau\|_{L^\infty(\mathbb{R}, H^{\frac{\alpha}{\alpha+2}, \kappa}(\mathbb{R}^n))} \\
    &\leq  \|u_0\|_{H^{\frac{\alpha}{\alpha+2}, \kappa}(\mathbb{R}^n)}+ \|u_1\|_{H^{-\frac{2}{\alpha+2}, \kappa}(\mathbb{R}^n)}+ \|F(u)\|_{L^1(\mathbb{R}, H^{-\frac{2}{\alpha+2}, \kappa}(\mathbb{R}^n))} \\
    &\lesssim \|u_0\|_{H^{\sigma(\alpha),\kappa}_\frac{\alpha+2}{\alpha+1}(\mathbb{R}^n)}+ \|u_1\|_{H^{\sigma(\alpha)-1,\kappa}_\frac{\alpha+2}{\alpha+1}(\mathbb{R}^n)}+ \|F(u)\|_{L^1(\mathbb{R}, H^{\sigma(\alpha)-1,\kappa}_\frac{\alpha+2}{\alpha+1})} \\
    &\lesssim \|u_0\|_{H^{\sigma(\alpha),\kappa}_\frac{\alpha+2}{\alpha+1}(\mathbb{R}^n)}+ \|u_1\|_{H^{\sigma(\alpha)-1,\kappa}_\frac{\alpha+2}{\alpha+1}(\mathbb{R}^n)}+ \||u|^{\alpha+1}\|_{L^1(\mathbb{R}, L_k^\frac{\alpha+2}{\alpha+1}(\mathbb{R}^n))} \\
    &=\|u_0\|_{H^{\sigma(\alpha),\kappa}_\frac{\alpha+2}{\alpha+1}(\mathbb{R}^n)}+ \|u_1\|_{H^{\sigma(\alpha)-1,\kappa}_\frac{\alpha+2}{\alpha+1}(\mathbb{R}^n)}+ \|u\|^\alpha_{L^{\alpha+1}(\mathbb{R}, L_k^{\alpha+2}(\mathbb{R}^n))},
\end{align*}
which indicates that the unique solution also belongs to $L^\infty(\mathbb{R}, H^{\frac{\alpha}{\alpha+2}, \kappa}(\mathbb{R}^n))$.
\end{proof}

Argued similarly as in the nonlinear Klein-Gordon equation \eqref{N-K-GEqu}, we also study the Cauchy problem for the nonlinear beam equation
\begin{equation}\label{N-BeamEqu}
	\begin{cases}
	\partial_t^2u+\Delta_\kappa^2 u+u=F(u),\\
	u|_{t=0}=u_0 ,  \\
	\partial_tu|_{t=0}=u_1.
	\end{cases}
	\end{equation}
By Duhamel's principle, the solution is formally given by
\begin{equation}\label{AA}
	u(t)=\frac{d\mathcal{B}_t}{dt}u_0+\mathcal{B}_tu_1-\int_0^t\mathcal{B}_{t-\tau}F(u)\,d\tau.
	\end{equation}
We also find the critical index
\begin{equation*}
    \alpha_B(N)=\frac{4-N+\sqrt{N^2+24N+16}}{2N}.
\end{equation*}
The following result is about the global well-posedness  for the nonlinear beam equation (\ref{AA}).
\begin{theorem}
     Let $\alpha_B(N)<\alpha<\frac{8}{N}$, $\sigma(\alpha)=\frac{\alpha N}{\alpha+2}-\frac{2}{\alpha+1}$, $\sigma(\alpha)<s\leq 2$, $s_0=s-\frac{\alpha  N}{2(\alpha+2)}$ and $(u_0,u_1)\in H^{s,\kappa}_\frac{\alpha+2}{\alpha+1}(\mathbb{R}^n)\times H^{s-2,\kappa}_\frac{\alpha+2}{\alpha+1}(\mathbb{R}^n)$ with sufficient small norm. Then the equation \eqref{AA} has a unique solution
     \begin{equation*}
         u\in C(\mathbb{R}, H^{s_0, \kappa}(\mathbb{R}^n))\cap L^{\alpha+1}(\mathbb{R},L_\kappa^{\alpha+2}(\mathbb{R}^n)).
     \end{equation*}
\end{theorem}
\begin{proof}
 From Proposition \ref{properties} and \ref{beam-decay}, we have
 \begin{align}
     \|\frac{d\mathcal{B}_t}{dt}u_0\|_{L^{\alpha+2}_k(\mathbb{R}^n)}&\leq  \|\frac{d\mathcal{B}_t}{dt} u_0\|_{B^{0,\kappa}_{\alpha+2,2}(\mathbb{R}^n)}\lesssim k(t) \|u_0\|_{B^{s,\kappa}_{\frac{\alpha+2}{\alpha+1},2}(\mathbb{R}^n)}\lesssim k(t) \|u_0\|_{H^{s,\kappa}_\frac{\alpha+2}{\alpha+1}(\mathbb{R}^n)},\nonumber\\ 
      \|\mathcal{B}_t u_1\|_{L^{\alpha+2}_k(\mathbb{R}^n)}&\leq  \|\mathcal{B}_t u_1\|_{B^{0,\kappa}_{\alpha+2,2}(\mathbb{R}^n)}\lesssim k(t) \|u_1\|_{B^{s-2,\kappa}_{\frac{\alpha+2}{\alpha+1},2}(\mathbb{R}^n)}\lesssim k(t) \|u_1\|_{H^{s-2,\kappa}_\frac{\alpha+2}{\alpha+1}(\mathbb{R}^n)},\label{B-decay}
 \end{align}
 where
 \begin{equation*}
 k(t)=\begin{cases}
	|t|^{-\max \left(\frac{\alpha N}{2(\alpha+2)}-\frac{s}{2},0\right)} , \, &|t|\leq1,\\
	|t|^{-\frac{\alpha N}{4(\alpha+2)}}, \,&|t|\geq1.
	\end{cases}
\end{equation*}
For $\alpha_B(N)<\alpha<\frac{8}{N}$ and  $\sigma(\alpha)<s\leq 2$, we see that 
\begin{equation*}
    (\alpha+1)\left(\frac{\alpha N}{2(\alpha+2)}-\frac{s}{2}\right)<1\quad \text{and}\quad (\alpha+1)\frac{\alpha N}{4(\alpha+2)}>1,
\end{equation*}
which implies that  $k\in L^{\alpha+1}(\mathbb{R})$ and 
\begin{align*}
  \|\frac{d\mathcal{B}_t}{dt} u_0\|_{L^{\alpha+1}(\mathbb{R},L^{\alpha+2}_k(\mathbb{R}^n))}&\leq C\|u_0\|_{H^{s,\kappa}_\frac{\alpha+2}{\alpha+1}(\mathbb{R}^n)},\\
 \|\mathcal{B}_t u_1\|_{L^{\alpha+1}(\mathbb{R},L^{\alpha+2}_k(\mathbb{R}^n))}&\leq C\|u_1\|_{H^{s-2,\kappa}_\frac{\alpha+2}{\alpha+1}(\mathbb{R}^n)}.
\end{align*}
Therefore,   from \eqref{B-decay}, Proposition \ref{properties}, and Young's inequality, it follows that
\begin{align*}
 \left\|\int_0^t\mathcal{B}_{t-\tau}F(u)\,d\tau\right\|_{L^{\alpha+1}(\mathbb{R},L^{\alpha+2}_k(\mathbb{R}^n))}&\leq  \left\|\int_0^t\|\mathcal{B}_{t-\tau}F(u)\|_{L^{\alpha+2}_k(\mathbb{R}^n)}\,d\tau\right\|_{L^{\alpha+1}(\mathbb{R})} \\
 &\leq C \left\|\int_0^t k(t-\tau)\|F(u)\|_{H^{s-2,\kappa}_\frac{\alpha+2}{\alpha+1}(\mathbb{R}^n)}\,d\tau\right\|_{L^{\alpha+1}(\mathbb{R})} \\
  &\leq C \left\|\int_0^t k(t-\tau)\|F(u)\|_{L_k^\frac{\alpha+2}{\alpha+1}(\mathbb{R}^n)}\,d\tau\right\|_{L^{\alpha+1}(\mathbb{R})} \\
  &\leq C\|k\|_{L^{\alpha+1}(\mathbb{R})} \|F(u)\|_{L^1(\mathbb{R},L_k^\frac{\alpha+2}{\alpha+1}(\mathbb{R}^n))}\\
  &\leq C\left\||u|^{\alpha+1}\right\|_{L^1(\mathbb{R},L_k^\frac{\alpha+2}{\alpha+1}(\mathbb{R}^n))}\\
  &=C\|u\|^{\alpha+1}_{L^{\alpha+1}(\mathbb{R},L_k^{\alpha+2}(\mathbb{R}^n))}.
\end{align*}
Taking $M=2C(\|u_0\|_{H^{s,\kappa}_\frac{\alpha+2}{\alpha+1}(\mathbb{R}^n)}+\|u_1\|_{H^{s-2,\kappa}_\frac{\alpha+2}{\alpha+1}(\mathbb{R}^n)})$,  we set the evolution spaces
\begin{equation*}
    X=\{u\in L^{\alpha+1}(\mathbb{R},L_k^{\alpha+2}(\mathbb{R}^n)): \|u\|_{L^{\alpha+1}(\mathbb{R},L_k^{\alpha+2}(\mathbb{R}^n))}\leq M\}.
\end{equation*}
Define the mapping
\begin{equation*}
    \mathcal{T}: u\mapsto \frac{d\mathcal{B}_t}{dt}u_0+\mathcal{B}_tu_1-\int_0^t\mathcal{B}_{t-\tau}F(u)\,d\tau, \;\forall u\in X.
\end{equation*}
Then for any $u,v\in X$, we have
\begin{equation*}
    \|\mathcal{T}u\|_{L^{\alpha+1}(\mathbb{R},L_k^{\alpha+2}(\mathbb{R}^n))}\leq \frac{M}{2}+CM^{\alpha+1}.
\end{equation*}
Thus
\begin{align*}
      &\quad\|\mathcal{T}u-\mathcal{T}v\|_{L^{\alpha+1}(\mathbb{R},L_k^{\alpha+2}(\mathbb{R}^n))}\\
     &=\left\|\int_0^t\mathcal{B}_{t-\tau}\left(F(u)-F(v)\right)\,d\tau\right\|_{L^{\alpha+1}(\mathbb{R},L_k^{\alpha+2}(\mathbb{R}^n))}\\
     &\leq C \left\|F(u)-F(v)\right\|_{L^1(\mathbb{R},L_k^\frac{\alpha+2}{\alpha+1}(\mathbb{R}^n))}\\
     &\leq C \left\|(|u|^\alpha+|v|^\alpha)|u-v|\right\|_{L^1(\mathbb{R},L_k^\frac{\alpha+2}{\alpha+1}(\mathbb{R}^n))}\\
    &=C\left\|\left\|(|u|^\alpha+|v|^\alpha)|u-v|\right\|_{L_k^\frac{\alpha+2}{\alpha+1}(\mathbb{R}^n)}\right\|_{L^1(\mathbb{R})}\\
     &\leq C\left\|\left\||u|^\alpha+|v|^\alpha\right\|_{L_k^\frac{\alpha+2}{\alpha}(\mathbb{R}^n)}\left\|u-v\right\|_{L_k^{\alpha+2}(\mathbb{R}^n)}\right\|_{L^1(\mathbb{R})}\\
      &\leq C\left\||u|^\alpha+|v|^\alpha\right\|_{L^\frac{\alpha+1}{\alpha}(\mathbb{R}, L_k^\frac{\alpha+2}{\alpha}(\mathbb{R}^n))}\left\|u-v\right\|_{L^{\alpha+1}(\mathbb{R}, L_k^{\alpha+2}(\mathbb{R}^n))}\\
       &\leq C\left(\left\||u|^\alpha\right\|_{L^\frac{\alpha+1}{\alpha}(\mathbb{R}, L_k^\frac{\alpha+2}{\alpha}(\mathbb{R}^n))}+\left\||v|^\alpha\right\|_{L^\frac{\alpha+1}{\alpha}(\mathbb{R}, L_k^\frac{\alpha+2}{\alpha}(\mathbb{R}^n))}\right)\left\|u-v\right\|_{L^{\alpha+1}(\mathbb{R}, L_k^{\alpha+2}(\mathbb{R}^n))}\\
       &=C\left(\|u\|_{L^{\alpha+1}(\mathbb{R}, L_k^{\alpha+2}(\mathbb{R}^n))}^\alpha+\|v\|_{L^{\alpha+1}(\mathbb{R}, L_k^{\alpha+2}(\mathbb{R}^n))}^\alpha \right)\left\|u-v\right\|_{L^{\alpha+1}(\mathbb{R}, L_k^{\alpha+2}(\mathbb{R}^n))}\\
       &\leq 2CM^\alpha \left\|u-v\right\|_{L^{\alpha+1}(\mathbb{R}, L_k^{\alpha+2}(\mathbb{R}^n))}.
\end{align*}
Notice that, for  $CM^\alpha<\frac{1}{2}$,     $\mathcal{T}: X\to X$ is a constraction mapping. Therefore, applying Banach's fixed point theorem, the equation \eqref{N-BeamEqu} has a unique
solution $u\in X$. In addition, by the embedding $ H^{s,\kappa}_\frac{\alpha+2}{\alpha+1}(\mathbb{R}^n)\subseteq H^{s_0,\kappa}(\mathbb{R}^n)$ and $ H^{s-2,\kappa}_\frac{\alpha+2}{\alpha+1}(\mathbb{R}^n)\subseteq H^{s_0-2,\kappa}(\mathbb{R}^n)$, we have
\begin{align*}
   & \quad\|u\|_{L^\infty(\mathbb{R}, H^{s_0, \kappa}(\mathbb{R}^n))}\\
    &\leq  \|\frac{d\mathcal{B}_t}{dt}u_0\|_{L^\infty(\mathbb{R}, H^{s_0, \kappa}(\mathbb{R}^n))}+ \|\mathcal{B}_tu_1\|_{L^\infty(\mathbb{R}, H^{s_0, \kappa}(\mathbb{R}^n))}+ \|\int_0^t\mathcal{B}_{t-\tau}F(u)\,d\tau\|_{L^\infty(\mathbb{R}, H^{s_0, \kappa}(\mathbb{R}^n))} \\
    &\leq  \|u_0\|_{H^{s_0, \kappa}(\mathbb{R}^n)}+ \|u_1\|_{H^{s_0-2, \kappa}(\mathbb{R}^n)}+ \|F(u)\|_{L^1(\mathbb{R}, H^{s_0-2, \kappa})} \\
    &\lesssim \|u_0\|_{H^{s,\kappa}_\frac{\alpha+2}{\alpha+1}(\mathbb{R}^n)}+ \|u_1\|_{H^{s-2,\kappa}_\frac{\alpha+2}{\alpha+1}(\mathbb{R}^n)}+ \|F(u)\|_{L^1(\mathbb{R}, H^{s-2,\kappa}_\frac{\alpha+2}{\alpha+1}(\mathbb{R}^n))} \\
    &\lesssim \|u_0\|_{H^{s,\kappa}_\frac{\alpha+2}{\alpha+1}(\mathbb{R}^n)}+ \|u_1\|_{H^{s-2,\kappa}_\frac{\alpha+2}{\alpha+1}(\mathbb{R}^n)}+ \||u|^{\alpha+1}\|_{L^1(\mathbb{R}, L_k^\frac{\alpha+2}{\alpha+1}(\mathbb{R}^n))} \\
    &=\|u_0\|_{H^{s,\kappa}_\frac{\alpha+2}{\alpha+1}(\mathbb{R}^n)}+ \|u_1\|_{H^{s-2,\kappa}_\frac{\alpha+2}{\alpha+1}(\mathbb{R}^n)}+ \|u\|^\alpha_{L^{\alpha+1}(\mathbb{R}, L_k^{\alpha+2}(\mathbb{R}^n))},
\end{align*}
which indicates that the unique solution also belongs to $L^\infty(\mathbb{R}, H^{s_0, \kappa}(\mathbb{R}^n))$.
\end{proof}

\section*{Statements and Declarations} The authors confirm that the data supporting the findings of this study are available within the article and its supplementary materials.

  \textbf{Competing Interests:} No potential competing of interest was reported by the author. 
\section*{Acknowledgments} C. Luo and M. Song are supported by the National Natural Science Foundation of China (Grant No. 11701452) and Guangdong Basic and Applied Basic Research Foundation (No. 2023A1515010656). S.S. Mondal is supported by the DST-INSPIRE Faculty Fellowship (Grant No. DST/INSPIRE/04/2023/002038).

\end{document}